\documentclass[12pt]{amsart}
\usepackage{amsmath}
\RequirePackage[usenames,dvipsnames]{color}
\usepackage{mathrsfs}
\usepackage{amssymb}
\usepackage[hidelinks]{hyperref}
\usepackage{stmaryrd}
\usepackage{graphicx}
\usepackage{enumitem} 
\usepackage{txfonts} 
\usepackage{comment}

\usepackage[top=1in, left=1in, right=1in, bottom=1in]{geometry}

\usepackage{todonotes}

\usepackage[normalem]{ulem}
\usepackage{xcolor}

\newcommand\Cstrike{\bgroup\markoverwith{\textcolor{magenta}{\rule[0.5ex]{2pt}{1pt}}}\ULon}

\newcommand\Gstrike{\bgroup\markoverwith{\textcolor{ForestGreen}{\rule[0.5ex]{2pt}{1pt}}}\ULon}

\newcommand\Bstrike{\bgroup\markoverwith{\textcolor{blue}{\rule[0.5ex]{2pt}{1pt}}}\ULon}

\usepackage[all,pdftex, cmtip]{xy}
\newdir{ >}{{}*!/-10pt/@{>}}

\usepackage{tikz-cd}
\usetikzlibrary{decorations.pathmorphing}
\usetikzlibrary{arrows}

\tikzset{%
	treenode/.style = {shape=rectangle, rounded corners,%
		draw, align=center,%
		top color=white, bottom color=blue!20},%
	root/.style     = {treenode, font=\Large, bottom color=red!30},%
	env/.style      = {treenode, font=\ttfamily\normalsize},%
	dummy/.style    = {circle,draw,inner sep=0pt,minimum size=2mm}%
}%

\DeclareMathAlphabet{\mathbbe}{U}{bbold}{m}{n}

\setlist{}
\setenumerate{leftmargin=*,labelindent=0\parindent}
\setitemize{leftmargin=*,labelindent=\parindent}

\newtheorem{thm}{Theorem}[section]

\newtheorem{prop}[thm]{Proposition}
\newtheorem{cor}[thm]{Corollary}

\theoremstyle{definition}
\newtheorem{defn}[thm]{Definition}
\newtheorem{ex}[thm]{Example}

\newtheorem{obs}[thm]{Observation}

\newtheorem{notation}[thm]{Notation}

\theoremstyle{remark}

\makeatletter
\@addtoreset{section}{part}
\makeatother  

\makeatletter
\let\c@equation\c@thm
\makeatother
\numberwithin{equation}{section}

\newcommand{\id}{\mathrm{id}}
\newcommand{\op}{\mathrm{op}}
\newcommand{\Ho}{\mathrm{Ho}}
\newcommand{\heq}{\mathrm{heq}}

\newcommand{\Mat}{\mathrm{Mat}}
\newcommand{\GL}{\mathrm{GL}}

\newcommand{\const}{\mathrm{const}}

\newcommand{\cC}{\mathcal{C}}
\newcommand{\cD}{\mathcal{D}}
\newcommand{\cE}{\mathcal{E}}
\newcommand{\cG}{\mathcal{G}}

\newcommand{\defeq}{\mathrel{:=}}
\newcommand{\eqdef}{\mathrel{=:}}

\newcommand{\Hom}{\text{Hom}}
\newcommand{\Map}{\text{Map}}
\newcommand{\map}{\text{map}}
\newcommand{\nerve}{\text{nerve}}

\newcommand{\iso}{\text{iso}}
\newcommand{\ob}{\text{ob}}
\newcommand{\ev}{\text{ev}}
\newcommand{\comma}{\text{,}}

\newcommand{\Set}{\texttt{Set}}
\newcommand{\SSet}{\texttt{SSet}}
\newcommand{\SSpace}{\texttt{SSpace}}
\newcommand{\FinSet}{\texttt{FinSet}}
\newcommand{\FinVect}{\texttt{Vect}}

\makeatletter
\def\makeslashed#1#2#3#4#5{#1{\mathpalette{\sla@{#2}{#3}{#4}}{#5}}}

\def\@mathlower#1#2#3{\setbox0=\hbox{$\m@th#2#3$}\lower#1\ht0\box0}
\def\mathlower#1#2{\mathpalette{\@mathlower{#1}}{#2}}
\makeatother

\ifx\color@rgb\@undefined\else
\definecolor{maroon}{rgb}{0.5,0,0}
\fi
\ifx\color@rgb\@undefined\else
\definecolor{violet}{rgb}{0.3,0,0.7}
\fi

\begin{document}

	\title[Decomposing the classifying diagram]{Decomposing the classifying diagram in terms of classifying spaces of groups}
	\author[Osborne]{Christina Osborne}


	\address{Department of Mathematics\\Ohio State University\\100 Math Tower\\231 West 18th Avenue\\Columbus OH, 43210-1174}
	\email{osborne.475@osu.edu}

	\date{\today}

	\begin{abstract}

	
	The classifying diagram was defined by Rezk and is a generalization of the nerve of a category; in contrast to the nerve, the classifying diagram of two categories is equivalent if and only if the categories are equivalent. In this paper we prove that the classifying diagram of any category is characterized in terms of classifying spaces of stabilizers of groups.  We also prove explicit decompositions of the classifying diagrams for the categories of finite ordered sets, finite dimensional vector spaces, and finite sets in terms of classifying spaces of groups. 
	
	\end{abstract}

	\maketitle
	
	\setcounter{secnumdepth}{2}
	\setcounter{tocdepth}{2}
	\tableofcontents
	\thispagestyle{myheadings}

	\section{Introduction}
	
	A topological space can be built from a category using the machinery of the nerve, which takes objects in the category to points and chains of $n$-composable morphisms to $n$-cells.
	The resulting space is referred to as the ``classifying space" of the category.  However, two categories that are not equivalent can produce equivalent classifying spaces because the nerve does not place any value on the data that comes from a morphism being an isomorphism. (See Example \ref{example_nerve_trivandnontrivcats} below.)
	
	The classifying diagram, which is a generalization of the nerve, is an alternative machine that can be used to study categories.
	The classifying diagram respects the data that isomorphisms provide, and as a result, the classifying diagram of two categories are equivalent if and only if the categories are equivalent \cite[3.3.4]{Bergner_book}. The classifying diagram was defined by Rezk in \cite{Rezk}. Rezk uses complete Segal spaces as a model for homotopy theory. The classifying diagram of a category is a natural occurrence of a complete Segal space. (See Proposition \ref{Prop_Classfying_Diagram_compl_Segal_space} below.)
	
	The purpose of this paper is to provide a deeper understanding of the classifying diagram.  We consider specific well-known categories, such as the category of finite ordered sets, finite dimensional vector spaces, and finite sets; we describe the resulting classifying diagrams in terms of classifying spaces of groups.  In the process, we prove the valuable result that for a general category we can characterize the classifying diagram in terms of classifying spaces of stabilizers.

	\subsection{Organization of the paper}
	
	We begin in Section \ref{Background} by recalling relevant category theory tools, the definition and basic structure of simplicial sets, the nerve of a category, and the definition of complete Segal spaces. In Section \ref{Classifying_Diagram_section}, we provide Rezk's definition for the classifying diagram of a category, explain why the classifying diagram is a complete Segal space, provide some preliminary examples of the classifying diagram, and prove a characterization of the classifying diagram using stabilizers of groups. In Sections \ref{Section_FinVect} and \ref{Section_FinSet} we address the classifying diagram of the categories of finite dimensional vector spaces and finite sets, respectively.

	\subsection{Acknowledgments}\label{Acknowlegments}
	
	 The content of this paper comes from my Ph.D. thesis \cite{Osborne}, which was completed at the University of Virginia. I want to thank my advisor, Julie Bergner, for her mentorship and encouragement throughout this project.  I also want to thank Nick Kuhn and Tom Mark for their helpful conversations and feedback.

	\section{Background}\label{Background}
	
	In this section we review some relevant category theory tools and provide an overview of simplicial sets, simplicial spaces, and complete Segal spaces.
	
	

	\subsection{Category theory tools}\label{NatTrans_as_a_functor}
	
	We recall some of the category theory tools that will be used throughout this paper. In particular, we review definitions and relevant results for natural transformations, equivalent categories, and diagram categories.

	The data from a natural transformation $\eta:F\Rightarrow G$ between two functors $F,G:\cC \rightarrow \cD$ can be equivalently packaged 
	as a functor $\eta: \cC \times \{0\rightarrow 1\} \rightarrow \cD$. First, recall the definition of a natural transformation.
	
	\begin{defn}
		A \emph{natural transformation} $\eta:F\Rightarrow G$ is a function that assigns to each object $c$ in $\cC$ a morphism $\eta_c:F(c)\rightarrow G(c)$ of $\cD$ in such a way that for every morphism $f:c\rightarrow c'$ of $\cC$, the diagram
		\[ 
		\begin{tikzcd}
		F(c) \ar[r, "\eta_c"] \ar[d, "F(f)"'] & G(c) \ar[d, "G(f)"] \\ F(c') \ar[r, "\eta_{c'}"] & G(c')
		\end{tikzcd}
		\]
		commutes.
	\end{defn}
	
	Instead of using the labeling for the category with two objects and one nontrivial morphism $\{0\rightarrow 1\}$, we suggestively use $\{F\xrightarrow{\eta} G\}$. First we show that given a natural transformation $\eta: F\Rightarrow G$, we can obtain a functor $\{F\xrightarrow{\eta} G\} \times \cC \rightarrow \cD$.  We define the evaluation functor $\ev:\{F\xrightarrow{\eta} G\} \times \cC \rightarrow \cD$ by $\ev (F,c)=F(c)$ and $\ev (G,c)=G(c)$.  The diagram
	\[ 
	\begin{tikzcd}
	(F,c) \ar[r, "(\eta \comma \id_c)"] \ar[d, "(\id_F \comma f)"'] & (G,c) \ar[d, "(\id_G \comma f)"] \\ (F,c') \ar[r, "(\eta\comma \id_{c'})"] & (G,c')
	\end{tikzcd}
	\]
	commutes in the category $\{F\xrightarrow{\eta} G\}\times \cC$.  Since functors preserve composition, applying the evaluation functor $\ev$ to the above diagram gives us the same diagram from the definition of natural transformations.  In this manner, we obtain the functor $\ev:\{F\xrightarrow{\eta} G\} \times \cC \rightarrow \cD$ 
	from a natural transformation $\eta:F\Rightarrow G$.
	
	The converse is also true.  Meaning, given functors $F,G:\cC \to \cD$ and $\ev\colon\{F\xrightarrow{\eta} G\} \times \cC \rightarrow \cD$ where $\ev (F,c)=F(c)$ and $\ev (G,c)=G(c)$, then we obtain a natural transformation $\eta:F\Rightarrow G$. To see that the converse is true, let $f:c\to c'$ be a morphism in $\cC$. Then the square
	\[ 
	\begin{tikzcd}
	(F,c) \ar[r, "(\eta \comma \id_c)"] \ar[d, "(\id_F \comma f)"'] & (G,c) \ar[d, "(\id_G \comma f)"] \\ (F,c') \ar[r, "(\eta\comma \id_{c'})"] & (G,c')
	\end{tikzcd}
	\]
	commutes in the category $\{F\xrightarrow{\eta} G\} \times \cC$. Applying the functor $\ev$ to the above diagram gives the commutative diagram
	\[ 
	\begin{tikzcd}
	F(c) \ar[r, "\ev(\eta \comma \id_c)"] \ar[d, "F(f)"'] & (G,c) \ar[d, "G(f)"] \\ (F,c') \ar[r, "\ev(\eta\comma \id_{c'})"] & (G,c')
	\end{tikzcd}
	\]
	in $\cD$. Thus if we define a function that assigns to each object $c$ of $(\cC)$ the morphism $\eta_c\defeq \ev(\eta,id_c):F(c)\to G(c)$ in $\cD$, we construct a natural transformation $\eta:F\Rightarrow G$.
	
	For functors $F,G:\cC\to\cD$, we say that a natural transformation $\eta: F\Rightarrow G$ is a \emph{natural isomorphism} if the morphism $\eta_c:F(c)\to G(c)$ in $\cD$ is an isomorphism for every object $c$ of $\cC$.
	
	\begin{defn}
		Categories $\cC$ and $\cD$ are \emph{equivalent categories} if there exist functors $F\colon\cC\to\cD$ and $G:\cD\to\cC$ as well as natural isomorphisms $G\circ F \cong id_{\cC}$ and $F\circ G \cong id_{\cD}$.
	\end{defn}

	We can alternatively determine if two categories are equivalent using the following definitions.
	
	\begin{defn}
		Let $F:\cC\to \cD$ be a functor.
		\begin{enumerate}
			\item If the function between hom-sets $F:\Hom_{\cC}(c,c') \to \Hom_{\cD}(F(c),F(c'))$ is injective for any objects $c,c'\in\cC$, then $F$ is \emph{faithful}.
			\item If the function between hom-sets $F:\Hom_{\cC}(c,c') \to \Hom_{\cD}(F(c),F(c'))$ is surjective for any objects $c,c'\in\cC$, then $F$ is \emph{full}.
			\item If for any object $d\in\cD$ there exists an object $c\in\cC$ such that there is an isomorphism $F(c)\xrightarrow{\cong} d$ in $\cD$, then $F$ is \emph{essentially surjective}.
		\end{enumerate}
	\end{defn}
	
	The following result says that there are necessary and sufficient requirements for a functor to define an equivalence of categories.
	
	\begin{prop}\cite[IV.4.1]{MacLane}
		The categories $\cC$ and $\cD$ are equivalent if and only if there exists a functor $F:\cC\to \cD$ that is faithful, full, and essentially surjective.
	\end{prop}
	
	We can use categories to define new categories. If the collection of objects in the category $\cD$ form a set, then we say that $\cD$ is a \emph{small} category. 
	
	\begin{ex}
		Let $\cC$ be a category and let $\cD$ be a small category.
		\begin{enumerate}
			\item 	The \emph{opposite category} of $\cC$, denoted $\cC^{op}$, is the category where $\ob({\cC})=\ob({\cC^{op}})$ and $f^{op}:b\to a$ is a morphism in $\cC^{op}$ if and only if $f:a\to b$ is a morphism in $\cC$.
			\item The \emph{functor category} $\cC^\cD$ is the category whose objects are functors $F:\cD\to \cC$ and the morphisms are natural transformations. We also refer to $\cC^\cD$ as a \emph{diagram category}.
		\end{enumerate}
	\end{ex}
	
	Let $[n]$ be the category consisting of a chain of $n$ composable morphisms\[0\to 1 \to 2 \to \cdots \to n.\] The diagram category $\cC^{[n]}$ is of particular importance throughout this paper.  An object of $\cC^{[n]}$ is a chain of $n$ composable morphisms in $\cC$
	\[ 
	\begin{tikzcd}
	x_0 \ar[r, "f_1"] & x_1 \ar[r, "f_2"] & x_2  \ar[r, "f_3"] &  \cdots \ar[r, "f_n"] & x_n
	\end{tikzcd}
	\]
	and a morphism from $(f_1,\ldots, f_n)$ to $(g_1,\ldots,g_n)$ is an $(n+1)$-tuple of morphisms $(\alpha_0, \alpha_1, \ldots, \alpha_n)$, where each $\alpha_i$ is a morphism in $\cC$, making the diagram
	\[ 
	\begin{tikzcd}
	x_0 \ar[r, "f_1"] \ar[d, "\alpha_0"] & x_1 \ar[r, "f_2"] \ar[d, "\alpha_1"] & x_2  \ar[r, "f_3"] \ar[d, "\alpha_2"] &  \cdots \ar[r, "f_n"] & x_n \ar[d, "\alpha_n"]\\
	y_0 \ar[r, "f_1"] & y_1 \ar[r, "f_2"] & y_2  \ar[r, "f_3"] &  \cdots \ar[r, "f_n"] & y_n
	\end{tikzcd}
	\]
	commute in $\cC$. If each $\alpha_i$ is an isomorphism in $\cC$, then $(\alpha_0,\ldots,\alpha_n)$ is an isomorphism in $\cC^{[n]}$.
	
	\begin{prop}\label{Equiv_of_functor_cats}
		Let $\cD$ and $\cE$ be equivalent small categories. Also let $\cC$ be a category. Then the functor categories $\cC^{\cD}$ and $\cC^{\cE}$ are equivalent. 
	\end{prop}
	
	\begin{proof}
		Let $F:\cD \to \cE$ and $G:\cE\to \cD$ be functors such that $G\circ F\cong \id_\cD$ and $F\circ G\cong \id_\cC$. Define $\overline{F}:\cC^\cE \to \cC^\cD$ by $\overline{F}(f)=f\circ F$ where $f$ is an object in $\cC^\cE$, and define $\overline{G}:\cC^\cD \to \cC^\cE$ by $\overline{G}(g)=g\circ G$ where $g$ is an object in $\cC^\cD$. Then
		\[\overline{G}\circ \overline{F}(f)=\overline{G}(f\circ F)=(f\circ F)\circ G= f\circ (F\circ G)\cong f\circ \id_\cE=f \]
		and similarly $\overline{F}\circ \overline{G}(g)\cong g$. Thus $\overline{F}$ and $\overline{G}$ define the desired equivalence of categories.
	\end{proof}
	
	\begin{prop}
		Let $\cC$ and $\cD$ be equivalent categories and let $\cE$ be a small category.  Then the functor categories $\cC^\cE$ and $\cD^\cE$ are equivalent.
	\end{prop}
	
	The proof of the above proposition is similar to the proof of Proposition \ref{Equiv_of_functor_cats}.
	
	
	\subsection{Simplicial sets} \label{ModelCat_SimpSet}
	
	A brief account of the definition of simplicial sets and a description of its model category structure are included here. 
	
	
	Before we can provide the definition of a simplicial set, we define the categories $\Set$ and $\Delta$. Let $\Set$ denote the category whose objects are sets and morphisms are functions. Let $\Delta$ be the category whose objects are finite ordered sets, denoted as $[n]=\{0\leq 1\leq 2 \leq \cdots \leq n\}$, and the morphisms are order-preserving functions.
	
	\begin{defn}
		A \emph{simplicial set} is a functor from $\Delta^{op}\to \Set$.
	\end{defn}
	
	If $X$ is is a simplicial set, we denote its geometric realization by $\vert X \vert$ \cite[\S I.2]{GoerssJardine}. The category of simplicial sets, which we denote by $\SSet$, has simplicial sets as objects and natural transformations for morphisms.
	
	\begin{ex}
		The \emph{standard $n$-simplex} is the simplicial set $\Delta[n]\defeq\Hom_{\Delta}(-,[n])$. The geometric realization $\vert\Delta [n]\vert$ is an $n$-cell.
	\end{ex}
	
	The data from a simplicial set $X$ can be rewritten in terms of sets $X([n])\eqdef X_n$ along with maps
	\[
	\begin{tabular}{l l l}
	$d_i:X_n\to X_{n-1}$, & $0\leq i \leq n$ & (face maps)\\
	$s_j:X_n\to X_{n+1}$, & $0\leq j \leq n$ & (degenercy maps)
	\end{tabular}
	\]
	which satisfy simplicial identities \cite[\S I.1]{GoerssJardine}.
	We say that $X_n$ is the $n$th level of the simplicial set $X$. By the Yoneda Lemma, the set $X_n\cong \Hom_\SSet(\Delta[n],X)$.

	There is a model structure for $\SSet$ where a map $f:X\to Y$ is a weak equivalence if the induced map from geometric realization, $\vert f \vert:\vert X \vert \to \vert Y \vert$, is a weak homotopy equivalence \cite[Ch.~1]{GoerssJardine}. 
	An expository account for model categories may be found in \cite{DwyerSpalinski}. The following proposition says that the hom-sets in $\SSet$ have the structure of a simplicial set.
	
	\begin{prop} \cite[II.2.2]{GoerssJardine}
		The category $\SSet$ is enriched in $\SSet$. That is, given any two simplicial sets $X$ and $Y$, the hom-set $\Hom_{\SSet}(X,Y)$ is a simplicial set.
	\end{prop}
	
	

	\subsection{The nerve of a category}
	
	In this section, we see that the nerve of a category gives a simplicial set.  Using the data that a natural transformation encodes, we show that nerves of categories are weakly equivalent if there exists functors between the categories with appropriate natural transformations with the identity functors.
	
	
	\begin{defn}
		The \emph{nerve} of a 
		category $\cC$ is the simplicial set defined levelwise by \[\nerve(\cC)_n\defeq \Hom_{Cat}([n],\cC). \]
	\end{defn}

	
	\begin{prop}\label{Nat_transf_homotopy}
		Given two functors $F,G: \cC \rightarrow \cD$ and a natural transformation $\eta: F\Rightarrow G$, there exists an induced homotopy $\vert\nerve (F)\vert \simeq \vert\nerve (G)\vert$.
	\end{prop}
	
	\begin{proof}
		As described in Section \ref{NatTrans_as_a_functor}, the notion of a  natural transformation $\eta : F\Rightarrow G$ equivalent to a functor $\eta: \cC \times  \{0\rightarrow 1\} \rightarrow \cD$. More specifically, we can think of $\eta : F\Rightarrow G$ as giving the commutative diagram
		\[
		\begin{tikzcd}
		\cC \times \{0\} \ar[d, hook] \ar[dr, bend left, "F"] &\\
		\cC \times  \{0\rightarrow 1\} \ar[r, "\eta"] & \cD. \\ 
		\cC \times \{1\} \ar[u, hook] \ar[ur, bend right, "G"'] & 
		\end{tikzcd}
		\]
		The functor $\vert\nerve (-)\vert$ can now be applied to this diagram.
		Note that
		$\vert\nerve (\{0\rightarrow 1\})\vert\simeq I$ where $I$ is the unit interval $[0,1]$. Also note that  $\vert\nerve (\cC \times \{ 0\to 1\}) \vert\simeq \vert \nerve(\cC)\vert \times I$ because $\vert \nerve(-)\vert$ preserves products \cite[14.1.5, 13.1.12]{Hirschhorn}. As a result, the diagram
		\[
		\begin{tikzcd}
		\vert\nerve (\cC )\vert \times 0 \ar[d] \ar[dr, bend left, "\vert \nerve(F)\vert"] &\\
		\vert\nerve (\cC) \vert\times I \ar[r, ""] & \vert\nerve (\cD)\vert \\ 
		\vert\nerve (\cC )\vert \times 1 \ar[u] \ar[ur, bend right, "\vert \nerve(G)\vert"'] & 
		\end{tikzcd}
		\]
		commutes and hence $\vert\nerve (\eta)\vert$  induces the desired 
		homotopy, $\vert\nerve (F)\vert \simeq \vert\nerve (G)\vert$.
	\end{proof}
	
	\begin{cor}\label{EquivCats_nerve}
		If the categories $\cC$ and $\cD$ are equivalent, then the simplicial sets $\nerve (\cC)$ and $\nerve (\cD)$ are weakly equivalent in the model structure for $\SSet$.
	\end{cor}

	
	\begin{proof}
		Suppose $\cC$ and $\cD$ are equivalent categories.  
		Then there exist functors $F:\cC \rightarrow \cD$ and $G:\cD \rightarrow \cC$ such that $G\circ F \cong \id_\cC$ and $F\circ G\cong \id_\cD$. By Proposition \ref{Nat_transf_homotopy}, the diagrams
		\[
		\begin{tikzcd}
		\vert\nerve (\cC )\vert \times \{0\} \ar[d] \ar[dr, bend left, "\vert\nerve (G\circ F)\vert"] & &
		\vert\nerve (\cD )\vert \times \{0\} \ar[d] \ar[dr, bend left, "\vert\nerve (F \circ G)\vert"] &\\
		\vert\nerve (\cC)\vert \times  I \ar[r] & \vert\nerve (\cC)\vert &
		\vert\nerve (\cD)\vert \times  I \ar[r] & \vert\nerve (\cD)\vert \\ 
		\vert\nerve (\cC )\vert \times \{1\} \ar[u] \ar[ur, bend right, "\vert\nerve (\id_\cC)\vert"'] & &
		\vert\nerve (\cD )\vert \times \{1\} \ar[u] \ar[ur, bend right, "\vert\nerve (\id_\cD)\vert"'] & 
		\end{tikzcd}
		\]
		commute. It should be noted that, for example, $\vert \nerve (G\circ F)\vert=\vert \nerve (G)\vert \circ \vert \nerve(F)\vert$ and $\vert \nerve (id_\cC)\vert=id_{\vert \nerve (\cC)\vert}$. 
		So we have homotopies $\vert\nerve(G\circ F)\vert\simeq id_{\vert \nerve(\cC)\vert}$ and $\vert\nerve(F\circ G)\vert\simeq id_{\vert \nerve(\cD)\vert}$ and hence we have a homotopy equivalence between $\vert \nerve (\cC)\vert$ and $\vert\nerve (\cD)\vert$.  Thus $\nerve (\cC)$ is weakly equivalent to $\nerve (\cD)$ in the model structure for $\SSet$.
	\end{proof}
	
	There is a more general result.
	
	\begin{prop}\label{Nat_transf_we_nerves}
		Given functors $F:\cC\to\cD$ and $G:\cD\to\cC$ with natural transformations $\eta:G\circ F\Rightarrow \id_\cC$ and $\theta: F\circ G \Rightarrow \id_\cD$, $\nerve(\cC)$ is weakly equivalent to $\nerve(\cD)$ in the model structure for $\SSet$.
	\end{prop}
	
	
	
	
	\begin{proof}
		By Proposition \ref{Nat_transf_homotopy}, we have homotopies $\vert\nerve (G\circ F)\vert\simeq  \id_{\vert \nerve(\cC)\vert}$ and  $\vert\nerve (F\circ G)\vert\simeq \id_{\vert \nerve(\cD)\vert}$. Thus $\vert \nerve (\cC)\vert$ is homotopy equivalent to $\vert \nerve (\cD)\vert$ and hence $\nerve(\cC)$ is weakly equivalent to $\nerve(\cD)$ in $\SSet$.
	\end{proof}
	
	Note that the direction of the natural transformations $\eta$ and $\theta$ had no effects on the proof for Proposition \ref{Nat_transf_we_nerves}.
	
	\subsection{A motivating example}\label{motivating_example}
	
	In Proposition \ref{Nat_transf_we_nerves} the categories $\cC$ and $\cD$ are not required to be equivalent in order for $\nerve(\cC)$ and $\nerve(\cD)$ to be weakly equivalent in $\SSet$. In fact, as we see in the following example, two categories can have weakly equivalent nerves even if the categories are not equivalent.
	
	\begin{ex}\label{example_nerve_trivandnontrivcats}
		Let $\cC$ be the category with one nontrivial morphism between two objects, $f:x\to y$, and let $\cD$ be the subcategory containing just the object $x$ and its identity morphism.  Let $F:\cC\to \cD$ be the functor sending every object of $\cC$ to $x$ and every morphism to the identity morphism on $x$. Let $G:\cD\to \cC$ be the natural inclusion. By construction, $F\circ G$ is the identity functor $\id_{\cD}$. Now we want to construct a natural transformation between $G\circ F$ and the identity functor $\id_\cC$. Note that $G\circ F(f)=\id_x$. We can construct a natural transformation $\eta:G\circ F \Rightarrow \id_\cC$ by letting $\eta_x\defeq \id_x$ and $\eta_y\defeq f$. Thus, by Proposition \ref*{Nat_transf_we_nerves}, the nerves of $\cC$ and $\cD$ are weakly equivalent. But we claim that the categories $\cC$ and $\cD$ are not equivalent.
		
		To see that $\cC$ and $\cD$ are not equivalent, note that the functor $F:\cC \to \cD$ is the only possible functor we can construct going from $\cC$ to $\cD$. Note that $\Hom_{\cC}(y,x)$ is the empty set, but $\Hom_{\cD}(F(y),F(x))=\Hom_{\cD}(x,x)=\{\id_x\}$. Thus $F$ is not faithful and hence the categories $\cC$ and $\cD$ are not equivalent.
	\end{ex}
	
	
	This example highlights the necessity for a finer tool than the nerve in order to distinguish the difference between categories that may not have been equivalent. The classifying diagram, which is defined in Section \ref{Classifying_Diagram_section},  is a generalization of the nerve and it can distinguish the difference between the categories described in the above example.  The purpose of this thesis is to provide a deeper investigation into the classifying diagram.
	
	\subsection{Simplicial spaces and the Reedy model structure}
	
	We provide the definition of a simplicial space, explain two different ways to build a simplicial space from a simplicial set, and describe what it means for simplicial spaces to be weakly equivalent.
	
	\begin{defn}
		A \emph{simplicial space} is a functor $\Delta^{\op}\to \SSet$.
	\end{defn}

	We denote the category of simplicial spaces by $\SSpace$. That is, $\SSpace$ has simplicial spaces as objects and natural transformations as morphisms. Because $\SSpace$ is enriched in $\SSet$, we let $\Map(X,Y)$ denote the the mapping space between $X$ and $Y$.
	
	The data for a simplicial space $X$ can be rewritten in terms of simplicial sets $X([n])\eqdef X_n$ along with face and degeneracy maps.
	
	Given a set $X$, we define the constant simplicial set by applying the functor $\const: \Set\to \SSet$ which maps the set $X$ to the simplicial set defined levelwise $\const(X)_n\defeq X$ where the face and degeneracy maps are identity maps.
	
	If instead we have a simplicial set $X$, we define two different simplicial spaces by applying two different functors $\SSet\to \SSpace$. 	
	\begin{enumerate}
		\item We define a constant simplicial space by applying the functor $\SSet\to \SSpace$ which maps the simplicial set $X$ to the simplicial space where each level is the simplicial set $X$  and the face and degeneracy maps are identity maps. We also denote the resulting constant simplicial space by $X$.
		\item Let $X_n$ be the $n$th level of the simplicial set $X$. We define a discrete simplicial space, denoted by $X^t$, by applying the functor $\SSet \to \SSpace$ levelwise; it maps the set $X_n$ to the simplicial set $X_n^t\defeq \const(X_n)$. 
	\end{enumerate}
	The ``$t$" in the notation of the simplicial space $X^t$  is used because  $X^t$ is an analog for the ``transpose" of the constant simplicial space $X$.

	
	
	The category $\SSpace$ is enriched in $\SSet$. The simplicial space $\Delta[n]^t$ is representable and hence we have \[X_n \cong \Map(\Delta[n]^t,X)\] using the enriched version of the Yoneda Lemma \cite[\S 2.3]{Rezk}. 
	
	
	In this paper we use the Reedy model category structure on $\SSpace$.  The weak equivalences in the Reedy model structure are the levelwise weak equivalences of simplicial sets \cite[A]{Reedy}. 
	


	\subsection{Segal spaces}

	In this section we see that a Segal space is a simplicial space with additional structure. In particular, if $X$ is a Segal space, then $X_n$ can be written in terms of $X_0$ and $X_1$.  
	
	In the category $\Delta$, define maps $\alpha^i:[1]\to[n]$ for $0\leq i <n$ where $\alpha^i(0)=i$ and $\alpha^i(1)=i+1$. Using the construction of the $\alpha^i$'s, we construct the simplicial set
	\[\displaystyle{G(n)\defeq \bigcup_{i=0}^{n-1} \alpha^i\Delta[1]\subset \Delta[n]}.\]
	
	For a simplicial set $X$  and $n\geq 2$,
	\begin{align*}
	\Hom_{\SSet} (G(n),X) & \cong \underbrace{ X_1\times_{X_0}\cdots \times_{X_0} X_1}_n\\
	& = \lim \left( X_1 \xrightarrow{d_0} X_0\xleftarrow{d_1} X_1  \xrightarrow{d_0} \cdots \xleftarrow{d_1} X_1\right).
	\end{align*}

	
	The inclusion $G(n)\subset \Delta[n]$ of simplicial sets induces an inclusion $G(n)^t \hookrightarrow \Delta[n]^t$ of simplicial spaces. For a fixed simplicial space $X$, this inclusion of simplicial spaces induces a map
	\[ 
	\begin{tikzcd}
	\Map(\Delta[n]^t, X) \ar[r] \ar[d, equals] & \Map(G(n)^t, X) \ar[d, equals] \\ X_n \ar[r, "\varphi_n"] & \underbrace{X_1 \times_{X_0}\cdots \times_{X_0} X_1}_n
	\end{tikzcd}
	\]
	between simplicial sets for $n\geq 2$. We call  $\varphi_n$ 
	a \emph{Segal map}.
	
	
	\begin{defn}\cite[\S 4.1]{Rezk}
		A simplicial space $W$ is a \emph{Segal space} if $W$ is Reedy fibrant and the Segal maps $\varphi_n$ are weak equivalences for $n\geq 2$.
	\end{defn}
	
	Generally speaking, in a model structure, every object is weakly equivalent to a fibrant object. In the definition of a Segal space, we require $W$ to be Reedy fibrant in order to guarantee pullbacks and homotopy pullbacks coincide. We also note that the Segal maps $\varphi_n$ in a Segal space are acyclic fibrations \cite[\S 4.1]{Rezk}.
	
	
	\subsection{Complete Segal spaces}
	
	In order to arrive at the definition of a complete Segal space, we observe that a Segal space mimics the structure of a category.  
	
	\begin{defn} \cite[\S 5.1]{Rezk}
		The \emph{set of objects} of a Segal space $W$ is $\ob(W)\defeq W_{0,0}$, which is the zeroth level of the simplicial set $W_0$.
	\end{defn}
	
	Now that a Segal space has objects like a category, we need to define the analog of the hom-space between two objects.
	
	\begin{defn} \cite[\S 5.1]{Rezk}
		Let $W$ be a Segal space and $x,y\in\ob (W)$. The \emph{mapping space} $\map_W(x,y)$ is defined by the pullback square
		\[ 
		\begin{tikzcd}
		\map_W(x,y) \ar[r] \ar[d] & W_1 \ar[d, "(d_1\comma d_0)"] \\ 
		\{(x,y)\} \ar[r] & W_0\times W_0
		\end{tikzcd}
		\]
		of simplicial sets. 	
	\end{defn}
	
	The requirement that $W$ is Reedy fibrant means that the the map $(d_1,d_0):W_1\to W_0\times W_0$ is a fibration \cite[\S 5.1]{Rezk}. Hence $\map_W(x,y)$ is also a homotopy pullback in the above diagram.
	
	Given a Segal space $W$, note that if $x\in W_0$, then $s_0x\in W_1$ and $d_1s_0x=x=d_0s_0x$. So if $x$ is in the set of objects of $W$, then $s_0x\in \map(x,x)_0$, which leads to the following definition.
	
	\begin{defn}\cite[\S 5.1]{Rezk}
		Given a Segal space $W$ and $x\in\ob(W)$, the \emph{identity map} of $x$ is defined to be $id_x\defeq s_0x\in \map_W(x,x)_0$.
	\end{defn}
	
	In a Segal space, so far we have objects, mapping spaces between objects, and the identity map.  
	In a category, composition is unique. 
	However, if $f\in \map_W(x,y)_0$ and $g\in \map_W(y,z)_0$, what does ``$g\circ f$" mean, and is it in $\map_W(x,z)_0$? 
	To see how Rezk answered this question, we first define what it means to be homotopic in the mapping space, and then generalize the definition of the mapping space.
	
	\begin{defn}\cite[\S 5.3]{Rezk}
		Let $f,g\in\map_W(x,y)_0$ where $x$ and $y$ are objects in a Segal space $W$. We say $f$ and $g$ are \emph{homotopic}, denoted $f\simeq g$, if they lie in the same component of $\map_W(x,y)$.
	\end{defn}
	
	
	In order to get to the point where we can talk about composition, we need to be able to relate, for example, $\map_W(x,y)\times \map_W(y,z)$ with $\map_W(x,z)$. To accomplish this relationship, we generalize the definition of the mapping space between two objects and define a mapping space between a finite collection of objects.
	In particular, given $x_0,\ldots,x_n\in\ob(W)$, $\map_W(x_0,\ldots,x_n)$ is defined by the pullback square
	\[ 
	\begin{tikzcd}
	\map_W(x_0,\ldots,x_n) \ar[r] \ar[d] & W_n \ar[d, ""] \\ 
	\{(x_0,\ldots,x_n)\} \ar[r] & \underbrace{W_0\times\cdots \times W_0}_{n+1}.
	\end{tikzcd}
	\]
	
	Let $V\defeq \map_W(x_0,\ldots,x_n)$ where $W$ is a Segal space. Then the Segal map \[\varphi_n:V_n \xrightarrow{\simeq} \underbrace{V_1\times_{V_0}\cdots \times_{V_0}V_1}_n\]
	is actually the map \[\varphi_n: \map_W(x_0,\ldots, x_n) \xrightarrow{\simeq} \map_W(x_0,x_1)\times \map_W(x_1,x_2)\times \cdots \times \map_W(x_{n-1},x_n).\]
	
	Let $(f,g)\in \map_W(x,y)_0\times\map_W(y,z)_0$ where $W$ is a Segal space. We want to define what it means to ``compose" $g$ and $f$. Recall that the Segal map $\map_W(x,y,z)\to \map_W(x,y)\times \map_W(y,z)$ is an acyclic fibration and the map $\emptyset\to \{(f,g)\}$ is a cofibration in the model structure for $\SSet$. Thus, by the fourth model category axiom \cite[3.3]{DwyerSpalinski}, there exists a lift $\{(f,g)\}\to \map_W(x,y,z)$ making the diagram
	\[ 
	\begin{tikzcd}
	\emptyset \ar[r] \ar[d, hook] & \map_W(x,y,z) \ar[d, two heads, "\simeq"', "\varphi_2"] \\ 
	\{(f,g)\} \ar[r] \ar[ru, dotted] & \map_W(x,y)\times \map_W(y,z)
	\end{tikzcd}
	\]
	commute.	Thus we can define a \emph{composition} of $g$ and $f$ as a lift $k\in\map_W(x,y,z)$ of $(f,g)$ along the Segal map $\varphi_2$. The \emph{result} of a composition $k$ is $d_1k\in\map_W(x,z)_0$. A result is not unique, but a result is unique up to homotopy. We let $g\circ f$ denote a result of a composition.  In particular, as seen in the following proposition, a Segal space has a category theory structure up to homotopy.
	
	\begin{prop}\cite[5.4]{Rezk}\label{SegalSpace_assoc_identity}
		In a Segal space $W$, let $w,x,y,z\in\ob(W)$ and $(f,g,h)\in\map_W(w,x)_0\times \map_W(x,y)_0\times \map_W(y,z)_0$. Then
		\begin{enumerate}
			\item $(h\circ g)\circ f \simeq h \circ (g\circ f)$ and
			\item $f\circ \id_w\simeq f \simeq \id_x\circ f$.
		\end{enumerate}
	\end{prop}
	
	We can use the up-to-homotopy category structure in a Segal space to define a category.
	
	\begin{defn}\cite[\S 5.5]{Rezk}
		The \emph{homotopy category} of a Segal space $W$, denoted as $\Ho W$, has $\ob(W)$ as objects and $\Hom_{\Ho(W)}(x,y)=\pi_0\map_W(x,y)$.
	\end{defn}

	If $f\in\map_W(x,y)_0$, let $[f]\in \Hom_{\Ho(W)}(x,y)$ denote its associated equivalence class.
	
	\begin{defn}\cite[\S 5.5]{Rezk}
		Let $W$ be a Segal space with objects $x$ and $y$. We say $f\in\map_W(x,y)_0$ is a \emph{homotopy equivalence} if $[f]\in\Hom_{\Ho(W)}(x,y)$ is an isomorphism.
	\end{defn}
	
	In other words, $f\in\map_W(x,y)_0$ is a homotopy equivalence if there exists $g,h\in\map_W(y,x)_0$ such that $f\circ g\simeq id_y$ and $h\circ f\simeq \id_x$. Observe that Proposition \ref{SegalSpace_assoc_identity} implies $g\simeq h$. Also note that $\id_x\in \map_W(x,x)_0$ is a homotopy equivalence.
	
	The following result shows us that we can define a subspace using the homotopy equivalences in a Segal space.
	
	\begin{prop}\cite[5.8]{Rezk}
		If $[f]=[g]\in\Hom_{\Ho(W)}(x,y)$, then $f$ is a homotopy equivalence  if and only if $g$ is a homotopy equivalence in the Segal space $W$.
	\end{prop}
	
	So in a Segal space $W$, the \emph{space of homotopy equivalences} is defined as the subspace $W_\heq\subseteq W_1$ which consists of the components of $W_1$ whose $0$-simplices are homotopy equivalences. 
	
	Note that for any object $x$ in $W$, $s_0x\defeq\id_x$ is a homotopy equivalence and hence the degeneracy map $s_0: W_0\to W_1$ factors through $W_\heq$.
	
	\begin{defn} \cite[\S 6]{Rezk}
		A \emph{complete Segal space} $W$ is a Segal space such that the map $s_0:W_0\to W_\heq$ is a weak equivalence.
	\end{defn}
	
	To see the importance of requiring $s_0$ to be a weak equivalence, consider the category $I[1]$, which is given by 
	\[ 
	\begin{tikzcd}
	0 \ar[r, bend left, "\cong"] & 1. \ar[l, bend left, "\cong"]
	\end{tikzcd}
	\]
	We let $ij$ denote the isomorphism $i\to j$ in $I[1]$ where $i,j\in\{0,1\}$. Let $E[1]\defeq\nerve(I[1])$. Then $E[1]_0=\{0,1\}$ and $E[1]_1=\{00,11,01,10\}$. Note that the categories $[0]$ and $I[1]$ are equivalent. Now compare the Segal space $\Delta[0]^t$ and $E[1]^t$. Levelwise $\Delta[0]^t$ is contractible, but $E[1]^t$ is not levelwise contractible.  For example, $E[1]^t_0$ is the constant simplicial set given by the set $\{0,1\}$. By definition, $\ob(E[1]^t)\defeq E[1]^t_{0,0}=\{0,1\}$. Note that $01\in \map_{E[1]^t}(0,1)$ and $10\in \map_{E[1]^t}(1,0)$ are homotopy equivalences.  So the objects $0$ and $1$ have homotopy equivalences going between them. However the simplicial set $E[1]^t_0$ is discrete; there is no path in $E[1]^t_0$ between $0$ and $1$. So in the definition of complete Segal spaces, the whole point of requiring $s_0:W_0\to W_\heq$ to be a weak equivalence is to guarantee that if there is a homotopy between two objects, then there is also a path between the objects.
	
	
	In the next section we show that the classifying diagram is a complete Segal space.

	\section{The classifying diagram}\label{Classifying_Diagram_section}

	
	
	In this section we see that Rezk's classifying diagram, which is a generalization of the nerve, is a simplicial space that naturally has the structure of a complete Segal space.  
	In Example \ref{example_nerve_trivandnontrivcats} we saw that the nerve fails at distinguishing between categories that are not equivalent; we revisit the categories from this example and show that the classifying diagrams are in fact not equivalent. We compute the classifying diagram of two preliminary examples: a finite ordered set $[n]$, and the category of finite ordered sets $\Delta$. 
	Additionally we show that the classifying diagram of a category $\cC$ is equivalent to the discrete simplicial space $\nerve(\cC)^t$ if and only if the identities are the only isomorphisms in $\cC$. We close the section by proving for a general category that the levels of the classifying diagram can be written in terms of classifying spaces of stabilizers. Every classifying diagram description we provide in this paper is decomposed into classifying spaces of groups.

	\subsection{The definition of the classifying diagram}
	
	Before we can define the classifying diagram, we need to specify the notation used in the definition. If $\cD$ is a category, let $\iso(\cD)$ denote the subcategory where the objects are the same as in $\cD$, but the morphisms are only the isomorphisms of $\cD$. In the literature, $\iso(\cD)$ is sometimes called the maximal subgroupoid of $\cD$.
	
	\begin{defn} \cite[\S 3.5]{Rezk}
		The \emph{classifying diagram} of a category $\cC$ is denoted by $NC$ and is the simplicial space defined levelwise by \[(N\cC)_n\defeq\nerve (\iso(\cC^{[n]})).\]
	\end{defn}
	
	\begin{prop}\cite[6.1]{Rezk}\label{Prop_Classfying_Diagram_compl_Segal_space}
		The classifying diagram of a small category $\cC$ is a complete Segal space.
	\end{prop}
	
	\begin{proof}
		We refer the reader to \cite[9.1.1]{Bergner_book} for the proof that $N\cC$ is Reedy fibrant. By construction, the Segal map
		\[\varphi_n:(N\cC)_n\to \underbrace{(N\cC)_1\times_{(N\cC)_0}\cdots \times_{(N\cC)_0}(N\cC)_1}_n\]
		is an isomorphism for all $n\geq 2$, and hence $N\cC$ is a Segal space. Now we need to show that $N\cC$ is complete. 
		Note that for any $x,y\in\ob(\cC)\cong \ob(N\cC)$, there is a natural bijection between the sets $\map_{N\cC}(x,y)_0$ and $\Hom_\cC (x,y)$. So if $f\in \left( (N\cC)_\heq\right)_0$ where $f\in\map_{N\cC}(x,y)_0$, then there exists $g\in \left((N\cC)_\heq\right)_0$ where $f\circ g=\id_y$, $g\circ f=\id_x$, and $g\in\map_{N\cC}(y,x)_0$. To see that $(N\cC)_{\heq}\simeq\nerve(\iso(\cC^{I[1]}))$, note that $f\in\left( (N\cC)_\heq\right)_0$ with inverse $g\in \map_{N\cC}(y,x)_0$ if and only if $(f,g)\in\ob\left(\iso(\cC^{I[1]})\right)\cong \nerve\left(\iso(\cC^{I[1]})\right)_0$. Since the categories $I[1]$ and $[0]$ are equivalent, we have that $\cC^{I[1]}\simeq \cC^{[0]}$ by Proposition \ref{Equiv_of_functor_cats}. Thus, using Corollary \ref{EquivCats_nerve}, we have
		\[(N\cC)_0\defeq\nerve \left(\iso(\cC^{[0]})\right) \simeq \nerve \left( \iso ( \cC^{I[1]})\right)\simeq(N\cC)_\heq \]
		and hence $N\cC$ is complete.
	\end{proof}
	
	The fact that the Segal maps
	\[\varphi_n:(N\cC)_n\to \underbrace{(N\cC)_1\times_{(N\cC)_0}\cdots \times_{(N\cC)_0}(N\cC)_1}_n\]
	are isomorphisms means that if we want to describe the $n$th level of the classifying diagram of a category $\cC$, it suffices to describe 
	$(N\cC)_0$ and $(N\cC)_1$ because $(N\cC)_n$ is $n$ copies of the $(N\cC)_1$ glued together along $(N\cC)_0$.
	 Thus in many examples presented in this paper, we only provide a description for the $0$th and $1$st levels.
	
	Let us revisit the motivating example from Section \ref{motivating_example}.
	
	\begin{ex}
		Let us consider the categories that were discussed in Example \ref{example_nerve_trivandnontrivcats} and justify the claim that the classifying diagram of these two categories are different. In order to show that the classifying diagrams of these two categories are not weakly equivalent in the Reedy model structure, it suffices to find one level in which the respective simplicial sets are not weakly equivalent in the model structure on simplicial sets. Let us begin with the subcategory $\cD$, which just has one object, $x$, and its identity morphism.  Note that $\iso(\cD^{[0]})\simeq\iso(\cD)=\cD$,  and hence \[(N\cD)_0\simeq \nerve(\cD).\] Now, for the category $\cC$, note that the morphism $f:x\to y$ is not an isomorphism and hence $\iso(\cC)$ only has the two objects, $x$ and $y$, along with their identity morphisms. So we have $\iso(\cC^{[0]})\simeq \iso(\cC)\simeq \cD \coprod \cD$, 
		and hence \[(N\cC)_0\simeq \nerve(\cD) \amalg \nerve(\cD).\] Therefore the classifying diagrams $N\cC$ and $N\cD$ are not weakly equivalent.
	\end{ex}

	\subsection{Categories with only isomorphisms}
	
	In contrast to the category $\cC$ in the previous example, we only consider categories whose morphisms are only isomorphisms in this section.
	We show that the classifying diagram of a category with only isomorphisms is levelwise equivalent to its classifying space.
	
	Any group $G$ can be thought of as a category.  Namely, $G$ is a category with one object whose morphisms are given by each element in the group.  Since each group element has an inverse, all of the morphisms in the category $G$ are isomorphisms.
	
	\begin{prop} \cite[\S 3.5]{Rezk} \label{Prop_Classifying_Diagram_Group}
		Let $G$ be a group thought of as a category with one object. Then 
		\[(NG)_n\simeq \nerve(G)\] for any $n\geq 0$.
	\end{prop}
	
	In other words, the classifying diagram of a group $G$ is levelwise equivalent to the classifying space of $G$. The proof of the above proposition follows from Proposition \ref{groupoid_thm}.	When describing the classifying diagram in the following sections, we write the levels in terms of classifying spaces of groups.

	Recall that a groupoid is a category in which every morphism is an isomorphism. We acquire the following general result.
	
	\begin{prop} \cite[\S 3.5]{Rezk} \label{groupoid_thm}  \label{Prop_Classifying_Diagram_Groupoid}
		Let $G$ be a groupoid. Then the classifying diagram of $G$ is levelwise equivalent to the classifying space of $G$.  That is, \[(NG)_n\simeq \nerve(G)\]for any $n\geq 0$.
	\end{prop}
	
	\begin{proof}
		Note that $\iso(G^{[n]})=G^{[n]}$ because every morphism in the category $G$ and hence $G^{[n]}$ is an isomorphism. By Corollary \ref{EquivCats_nerve}, it suffices to prove the categories $G$ and $G^{[n]}$ are equivalent. Here we prove that $G$ and $G^{[2]}$ are equivalent, which can be extended to the more general result. Define a functor $\iota:G\to G^{[2]}$  on objects by
		\[
		\begin{tikzcd} 
		x
		\end{tikzcd}
		\begin{tikzcd} 
		\left.\right. \ar[rr, mapsto] && \left.\right.
		\end{tikzcd}	
		\left(
		\begin{tikzcd}
		x \ar[r, "id"] & x \ar[r, "id"] & x
		\end{tikzcd}
		\right)
		\]
		and on morphisms by
		\[
		\begin{tikzcd} 
		x \ar[d, "f"] \\ y
		\end{tikzcd}
		\begin{tikzcd} 
		\left.\right. \ar[rr, mapsto] && \left.\right.
		\end{tikzcd}	
		\begin{tikzcd}
		x \ar[r, "id"] \ar[d, "f"] & x \ar[r, "id"] \ar[d, "f"] & x \ar[d, "f"]\\
		y \ar[r, "id"] & y \ar[r, "id"] & y.
		\end{tikzcd}
		\]
		It is a straightforward observation that $\iota$ is full and faithful,  so we only show that $\iota$ is essentially surjective. Let $x\xrightarrow{g}y \xrightarrow{h}z$ be an arbitrary object in $G^{[2]}$. Note that the diagram
		\[
		\begin{tikzcd}
		x \ar[r, "id"] \ar[d, "id"] & x \ar[r, "id"] \ar[d, "g"] & x \ar[d, "h\circ g"]\\
		x \ar[r, "g"] & y \ar[r, "h"] & z.
		\end{tikzcd}
		\]
		commutes and hence the triple $(id,g,h\circ g)$ defines a morphism from $\iota(x)$ to $x\xrightarrow{g}y \xrightarrow{h}z$ in $G^{[2]}$. In fact, $(id,g,h\circ g)$ is an isomorphism with inverse $(id, g^{-1}, g^{-1}\circ h^{-1})$ because $G$ is groupoid. Thus $\iota$ is essentially surjective.
		Therefore the categories $G$ and $G^{[n]}$ are equivalent.
	\end{proof}
	
	In the following sections we use the following notation.
	
	\begin{notation}
		\begin{enumerate}
			\item We use classifying space notation, $BG$, instead of writing $\nerve(G)$.
			\item Let $G$ and $H$ be categories. We use $G\coprod H$ to denote the category with subcategories $G$ and $H$ such that an object (resp. a morphism) is in $G\coprod H$ if and only if it is an object (resp. a morphism) of $G$ or $H$.
		\end{enumerate}
	\end{notation}

	\begin{obs}\label{B_commutes_with_disjoint_union}
		Let $G_i$ be a category for each $i$. Then \[ \displaystyle{B\left( \amalg_i G_i\right) \simeq \amalg_i B(G_i)}.\]
	\end{obs}
	
	The above observation follows from the construction of the $\nerve$ because if we have a chain of $n$ composable morphisms in the category $\amalg_i G_i$, then the chain of morphisms lies only in $G_i$ for some $i$.

	\begin{ex}
	Let $G$ be a group.  We can alternatively think of it as a category, $\cG$, where the objects are the elements of the group, and the morphisms are identities. Then the classifying diagram is levelwise given as \[\left(N\cG\right)_n\simeq \coprod_{\vert \ob(\cG) \vert} B(\{e\}). \]
	\end{ex}

	

	\subsection{Preliminary examples} 
	
	In this section we consider two categories in which the classifying diagram is decomposed into disjoint unions of the classifying space of the trivial group.

	\subsubsection{A finite ordered set}
	
	Consider the category $[m]$, which consists of $(m+1)$ objects $0,1,\ldots,m$ and there exists one morphism $j\to k$ if and only if $j\leq k$.
	
	\begin{prop}
		The $0$th level of the classifying diagram of the category $[m]$ is given by \[N([m])_0\simeq \coprod_{m+1} B(\{e\})\]
		where $\{e\}$ is the trivial group, and the $1$st level of the classifying diagram is given by \[N([m])_1\simeq \coprod_{\frac{1}{2}(m+1)(m+2)} B(\{e\}).\]
	\end{prop}
	
	\begin{proof}
		To compute the $0$th level, we need to consider $\iso([m])$. Since the only isomorphisms in $[m]$ are identities, $\iso([m])$ consists of $m+1$ objects and only the identity morphisms. Thus \[N([m])_0\simeq \coprod_{m+1} B(\{e\}).\]
		Now to compute the $1$st level, consider the functor category $\iso([m]^{[1]})$. The number of objects in this functor category is the same as the number of morphisms in the category $[m]$.  Observe that each object $k$ in $[m]$ has $k+1$ morphisms mapping into it, and hence the number of morphisms in the category $[m]$ is given by \[\sum_{k=0}^{m}(k+1)=\frac{(m+1)(m+2)}{2}.\]
		Since the only isomorphisms in $[m]$ are identities, the only isomorphisms in $[m]^{[1]}$ are also only given by identities. Thus \[N([m])_0\simeq \coprod_{\frac{1}{2}(m+1)(m+2)} B(\{e\}).\]
	\end{proof}
	
	\subsubsection{The category of finite ordered sets}
	Consider the category of finite ordered sets $\Delta$. That is, the objects are given by finite ordered sets $[m]=\{0\leq 1\leq 2\leq \cdots \leq m\}$ and the morphisms are order preserving functions.
	
	\begin{prop}\label{N_Delta}
		The $0$th and $1$st levels of the classifying diagram of $\Delta$ are 
		both weakly equivalent to a disjoint union of countably many contractible spaces,
		but $N\Delta$ is not weakly equivalent to the constant simplicial space $\coprod_\mathbb{N} B(\{e\})$.
	\end{prop}
	
	\begin{proof}
		Note that the only isomorphisms in $\Delta$ are the identity morphisms, hence, since the objects of $\Delta$ are in bijection with the natural numbers, we get \[N(\Delta)_0\simeq \coprod_\mathbb{N} B(\{e\}).\]
		For the first level of the classifying diagram, the objects in $\iso(\Delta^{[1]})$ are morphisms $f:[n]\to [m]$ in $\Delta$, and the morphisms in the category $\iso(\Delta^{[1]})$ are given by pairs of isomorphisms $(\alpha, \beta)$ making the diagram
		\[ 
		\begin{tikzcd}
		\left[n\right] \ar[r, "\alpha"] \ar[d, "f"'] & \left[n\right] \ar[d, "g"] \\
		\left[m\right] \ar[r, "\beta"] & \left[m\right]
		\end{tikzcd}
		\]
		commute in $\Delta$. Since the only isomorphisms in $\Delta$ are identities, the only morphisms in the category $\iso(\Delta^{[1]})$ are given by pairs of identities. For any given pair of natural numbers, $n$ and $m$, there is only a finite number of morphisms $f:[n]\to [m]$ in $\Delta$. Also, we have already observed that there are a countable number of objects in $\Delta$.  Thus there is a countable number of morphisms in $\Delta$ and hence we get our desired result
		\[N(\Delta)_1\simeq \coprod_{\mathbb{N}} B(\{e\}).\]
		
		It remains to show that $N\Delta$ is not the constant simplicial space $\coprod_{\mathbb{N}} B(\{e\})$. For each $[n]$, let $B(\{e\})_{[n]}$ denote the copy of $B(\{e\})$ in $N(\Delta)_0$ corresponding to $[n]$. Similarly, let $B(\{e\})_{f:[n]\to[m]}$ denote the copy of $B(\{e\})$ in $N(\Delta)_1$ corresponding to the morphism $f:[n]\to [m]$ in $\Delta$. The diagram 
		\[
		\begin{tikzcd}
		& B(\{e\})_{f:[n]\to[m]}\ar[dl,"d_1"'] \ar[dr,"d_0"] &\\
		B(\{e\})_{[n]} & & B(\{e\})_{[m]}
		\end{tikzcd}
		\]
		shows how the face maps $d_0,d_1:(N\Delta)_1\to (N\Delta)_0$ interact for a given morphism $f:[n]\to [m]$. In particular, since  $m$ can be any nonnegative integer, we have a countable (not finite) collection of copies of $B(\{e\})$ in $(N\Delta)_1$ that map via $d_1$ to $B(\{e\})_{[n]}$ for each $n$. Also note that $s_0(B(\{e\})_{[n]})=B(\{e\})_{\id:[n]\to [n]}$. Thus 
		 $N\Delta$ is not weakly equivalent to the constant simplicial space.
	\end{proof}
	
	In Proposition \ref{Prop_Classifying_Diagram_Groupoid}, we saw that the classifying space of a groupoid $G$ is the weakly equivalent to the constant simplicial space $BG$; in particular $(NG)_0\simeq (NG)_1\simeq BG$ and the face/degeneracy maps are essentially identities. In contrast, we just proved that the $0$th and $1$st levels of $N\Delta$ are both countable disjoint unions of contractible spaces, but the face/degeneracy maps are not essentially identities. It is not a surprise that $N\Delta$ is not a constant simplicial space because $\Delta$ is not a groupoid.
	
	\subsection{The classifying diagram and the discrete simplicial space given by the nerve}
	
	In Proposition \ref{N_Delta}, the structures of $0$th and $1$st levels of $N\Delta$ are reminiscent of the $0$th and $1$st levels of the discrete simplicial space $\nerve(\Delta)^t$. In the following proposition, we provide the necessary and sufficient conditions on a category $\cC$ to guarantee  $N\cC$ and $\nerve(\cC)^t$ are isomorphic simplicial spaces; as a consequence $N\Delta$ is isomorphic to $\nerve(\Delta)^t$. 

		\begin{prop}
		The classifying diagram of a category $\cC$ is isomorphic to the discrete simplicial space $\nerve(\cC)^t$ if and only if  $\iso(\cC)$ is discrete.
	\end{prop}
	
	\begin{proof}
	Recall that $\nerve(\cC)^t_n\defeq\const (\nerve(\cC)_n)$ is the simplicial set given by the set $\nerve(\cC)_n$ at each level in which the face and degeneracy maps are identities. The only morphisms in $\iso(\cC^{[n]})$ are identities if and only if the only isomorphisms in $\cC$ are identities.
		 Thus any functor $[m]\to \iso(\cC^{[n]})$ maps $[m]$ to a chain of length $m$ of identity morphisms for an object in $\iso(\cC^{[n]})$ if and only if the only isomorphisms in $\cC$ are identities;  hence $\Hom_{Cat}([m],\iso(\cC^{[n]}))\cong \ob(\iso(\cC^{[n]}))$ if and only if the only isomorphisms in $\cC$ are identities. Therefore
		\begin{align*}
		N(\cC)_{n,m}&\defeq \nerve(\iso(\cC^{[n]}))_m\\
		&=\Hom_{Cat}([m],\iso(\cC^{[n]}))\\
		&\cong \ob(\iso(\cC^{[n]}))\\
		&\cong \ob (\cC^{[n]})\\
		&\cong \Hom_{Cat}([n],\cC)\\
		&=\const(\Hom_{Cat}([n],\cC))_m\\
		&\eqdef \const (\nerve (\cC)_n)_m
		\end{align*}
		if and only the only isomorphisms in $\cC$ are identities, which gives the desired result.
	\end{proof}
	
	\subsection{The stabilizer characterization of the classifying diagram}\label{Section_stabilizer_characterization}
	
	In Proposition \ref{Bergner 7.2}, the $0$th level of the classifying diagram is written in terms of classifying spaces of automorphism classes of the category.  One way to extend Proposition \ref{Bergner 7.2} to the higher levels of the classifying diagram is to describe the higher levels in terms of stabilizers of products of automorphism groups.  We begin the section by recalling the definition of a stabilizer and the fact that it is a group.
	
	
	\begin{prop}\cite[II.4.2]{Hungerford}
		Let $G$ be a group that acts on a set $S$.
		\begin{enumerate}
			\item The relation on $S$ defined by 
			\[x\sim x' \iff g \cdot x=x' \mbox{ for some } g\in G \]
			is an equivalence relation.
			\item The stabilizer for some $x\in S$, $G_x=\{ g\in G \vert g\cdot x=x\}$, is a subgroup of $G$.
		\end{enumerate}
	\end{prop}

	In an category $\cC$, the automorphisms on an object form a group. The following proposition uses automorphisms and hom-sets in a category to define a group action.
	
	\begin{prop}\label{Prop_Group_action_bullet}
		Let $x_0,x_1\ldots,x_n$ be objects in the category $\cC$. Given the $(n+1)$-tuple $\underline{x}=(x_0,\ldots,x_n)$, consider the group \[Aut(\underline{x})\defeq Aut(x_0) \times Aut(x_1)\times \cdots Aut(x_n)\] and the set 
		\[\Hom(\underline{x})\defeq \Hom_{\cC}(x_0,x_1)\times \Hom_{\cC}(x_1,x_2)\times \cdots \times \Hom_{\cC}(x_{n-1},x_n).\] The map
		\[ 
		\begin{tikzcd}
		\bullet : Aut(\underline{x}) \times \Hom(\underline{x}) \ar[r]  & \Hom(\underline{x}) \\
		\left( (\alpha_0,\ldots,\alpha_n), (f_1,\ldots,f_n) \right) \ar[r, mapsto] & \left( \alpha_1f_1\alpha_0^{-1}, \ldots, \alpha_nf_n\alpha_{n-1}^{-1}  \right),
		\end{tikzcd}
		\] 
		where each $\alpha_i$ is an automorphism on $x_i$ and $f_i:x_{i-1}\to x_i$ is a morphism in $\cC$, defines a group action. 
	\end{prop}
	
	\begin{proof}
		It suffices to check that $\bullet$ is a group action on each coordinate of $\Hom(\underline{x})$.  In other words, we show that the map
		\[
		\begin{tikzcd}
		\cdot : \left( Aut(x_{i-1})\times Aut(x_i) \right) \times \Hom_{\cC}(x_{i-1}, x_i) \ar[r] & \Hom_{\cC}(x_{i-1}, x_i)\\
		\left( (\alpha_{i-1},\alpha_i), f_i \right) \ar[r, mapsto] & \alpha_if_i\alpha_{i-1}^{-1}
		\end{tikzcd}
		\] 
		defines a group action for $1\leq i \leq n$. Note that $(\id_{x_{i-1}},\id_{x_i})\cdot f_i=\id_{x_i} f_i \id_{x_{i-1}}=f_i$. So it remains to show the compatibility of the action. To see that the action is compatible, observe that
		\begin{align*}
		\left( (\alpha_{i-1},\alpha_i)(\beta_{i-1},\beta_i) \right) \cdot f_i =& (\alpha_{i-1}\beta_{i-1},\alpha_i\beta_i) \cdot f_i\\
		=& \left(\alpha_i\beta_i \right) f_i \left( \alpha_{i-1}\beta_{i-1} \right)^{-1}\\
		=& \alpha_i\beta_i f_i \beta_{i-1}^{-1} \alpha_{i-1}^{-1}\\
		=& \alpha_i \left[ ( \beta_{i-1},\beta_i) \cdot f_i\right] \alpha_{i-1}^{-1}\\
		=& (\alpha_{i-1},\alpha_i) \cdot \left[ ( \beta_{i-1},\beta_i) \cdot f_i\right].
		\end{align*}
		Thus, since $\bullet$ is defined coordinate-wise and we have shown that $\cdot$ defines the action on each coordinate, $\bullet$ defines an action of $Aut(\underline{x})$ on $\Hom(\underline{x})$.
	\end{proof}
	
	Before we use $\bullet$ to describe the higher levels of the classifying diagram, we first provide the known description of the $0$th level the served as inspiration.
	
	\begin{prop}\cite[\S 7.2]{Bergner} \label{Bergner 7.2}
		Let $\cC$ be a category. For a given object $x$, let $\langle x \rangle$ denote its isomorphism equivalence class. Then		
		\[N(\cC)_0\simeq \coprod_{\langle x \rangle} B(Aut(x)).\]
	\end{prop}

	\begin{proof}
		Note the categories $\iso(\cC^{[0]})$ and $\coprod_{\langle x \rangle} Aut(x)$ are equivalent. By Corollary \ref{EquivCats_nerve}, $N(\cC)_0\simeq B\left( \coprod_{\langle x \rangle} Aut(x) \right)$. Applying Observation \ref{B_commutes_with_disjoint_union} gives the desired result.
		
	\end{proof}
	
	Moving onto to higher levels, it was previously shown that the $1$st level of the classifying diagram can be written in terms of automorphisms of morphisms.
	
	\begin{prop} \cite[\S 7.2]{Bergner}
		Let $\cC$ be a category.  Then \[N(\cC)_1\simeq \coprod_{\langle x \rangle , \langle y \rangle} \coprod_{\langle \alpha:x\to y \rangle} B(Aut(\alpha))\] where $\langle \alpha:x\to y \rangle$ is the automorphism class of the morphism $\alpha:x\to y$ in $\cC$.
	\end{prop}
	
	We use the action $\bullet$ to form a new characterization of the higher levels for the classifying diagram by using stabilizers under the action.
	
	\begin{thm}\label{Stabilizer_perspective}
		%
		%
		Let $\langle f_1,\ldots,f_n \rangle$ denote the equivalence class of $( f_1,\ldots,f_n)\in \Hom(\underline{x})$ defined by the group action $\bullet$. Then for $n\geq 1$,
		\[ N(\cC)_n \simeq  \coprod_{\langle f_1,\ldots,f_n \rangle} B\left[ Aut(\underline{x})_{(f_1,\ldots,f_n)} \right] \]
		where $Aut(\underline{x})_{(f_1,\ldots,f_n)}$ is the stabilizer of $(f_1,\ldots,f_n)$.
	\end{thm}
	
	\begin{proof}
		
		It suffices to argue that the categories \[ \iso(\cC^{[n]}) \hspace{.4 cm} \mbox{ and } \hspace{.4 cm}  \coprod_{\langle f_1,\ldots,f_n \rangle}  Aut(\underline{x})_{(f_1,\ldots,f_n)} \] are equivalent.	 Note that the objects in $\iso (\cC^{[n]})$ are given by chains of $n$ composable morphisms $f_1,f_2,\ldots, f_n$ between objects $x_0,x_1,\ldots,x_n$ in $\cC$:
		\[
		\begin{tikzcd}
		x_0 \ar[r, "f_1"] & x_1 \ar[r, "f_2"] & \cdots \ar[r, "f_n"] & x_n.
		\end{tikzcd}
		\]
		We denote this object in $\iso(\cC^{[n]})$ by the $n$-tuple $(f_1,f_2,\ldots, f_n)$. A morphism between objects $(f_1,f_2,\ldots, f_n)$ and $(g_1,g_2,\ldots, g_n)$ in $\iso(\cC^{[n]})$ is given by an $(n+1)$-tuple $(\alpha_0,\alpha_1,\ldots,\alpha_n)$ where each $\alpha_i$ is an isomorphism in $\cC$ making the diagram
		\[
		\begin{tikzcd}
		x_0 \ar[r, "f_1"] \ar[d, "\alpha_0", "\cong"'] & x_1 \ar[r, "f_2"] \ar[d, "\alpha_1", "\cong"'] & \cdots \ar[r, "f_n"] & x_n \ar[d, "\alpha_n", "\cong"']\\
		y_0 \ar[r, "g_1"] & y_1 \ar[r, "g_2"] & \cdots \ar[r, "g_n"] & y_n
		\end{tikzcd}
		\]
		commute in $\cC$. As we saw in Proposition \ref{Bergner 7.2}, it suffices to consider the automorphism classes of objects in $\iso(C^{[n]})$. So, in other words, for a given $(f_1,\ldots,f_n)$, we need to describe all possible morphisms $(\alpha_1,\ldots, \alpha_n)$ that fix $(f_1,\ldots,f_n)$. But this is exactly what the stabilizer $ Aut (\underline{x}) _{(f_1,\ldots,f_n)}$ does. Thus, $\iso(\cC^{[n]})$ is equivalent to \[ \coprod_{\langle f_1,\ldots,f_n \rangle}  Aut(\underline{x})_{(f_1,\ldots,f_n)}. \]
		
	\end{proof}
	
	\section{The classifying diagram for the category of vector spaces}\label{Section_FinVect}
	
	In this section we prove that the classifying diagram of finite vector spaces can be written in terms of classifying spaces of general linear groups; we use the group action defined in Proposition \ref{Prop_Group_action_bullet} as well as the other results from Section \ref{Section_stabilizer_characterization}. We then produce a more detailed description if we restrict the dimension of the matrices.
	
	
	The category of finite dimensional vector spaces over the field $\mathbb{F}$ has finite vector spaces as objects and linear maps as morphisms. Recall that every finite dimensional vector space is isomorphic to $\mathbb{F}^n$. Let $\FinVect(\mathbb{F})$ denote the subcategory of finite vector spaces where the objects are $\mathbb{F}^n$ and the morphisms are matrices with entries in $\mathbb{F}$. Note that $\FinVect(\mathbb{F})$ is equivalent to the category of finite dimensional vector spaces over $\mathbb{F}$. We let $\Mat_{n\times m}(\mathbb{F})$ be the set of $n\times m$ matrices. Also let $\GL_n{(\mathbb{F})}$ denote the general linear group of dimension $n$.
	The following proposition is a well-known result in linear algebra and is a specific case of the action $\bullet$ from Proposition \ref{Prop_Group_action_bullet}, but it is used to prove Corollary \ref{FinVect_stabilizer_perspective}.

	\begin{prop}
		The map
		\[ 
		\begin{tikzcd}
		\bullet : \left(\GL_n (\mathbb{F}) \times \GL_m (\mathbb{F})\right)\times \Mat_{n\times m}(\mathbb{F}) \ar[r]  & \Mat_{n\times m}(\mathbb{F}) \\
		\left( (f,g), A\right) \ar[r, mapsto] & gAf^{-1}
		\end{tikzcd}
		\]
		defines a group action.
	\end{prop}
	
	\begin{proof}
		For simplicity, we refrain from referencing the underlying field $\mathbb{F}$ in this proof. Specifically, we write $\GL_n\defeq\GL_n(\mathbb{F})$ and $\Mat_{n\times m}\defeq \Mat_{n\times m}(\mathbb{F})$.  We need to verify two things to verify that $\bullet$ is a group action.  First we need to verify that the action of the identity element in $\GL_n\times \GL_m$ preserves any element of the set $\Mat_{n\times m}$.  Let $A\in\Mat_{n\times m}$ and let $I_n$ denote the identity $n\times n$ matrix.  So the identity element in the group $\GL_n\times \GL_m$ is given by $(I_n,I_m)$. Thus
		\begin{align*}
		(I_n, I_m)\bullet A &=I_mAI_n^{-1}\\
		&=I_mAI_n\\
		&=A.
		\end{align*}
		Next, the compatibility of the action must be verified. Let $(f_1,g_1)$ and $(f_2,g_2)$ be elements in $\GL_n\times \GL_m$. Therefore
		\begin{align*}
		\left( (f_1,g_1)(f_2,g_2)\right)\bullet A &= (f_1f_2,g_1g_2)\bullet A\\
		&=(g_1g_2)A(f_1f_2)^{-1}\\
		&=g_1g_2Af_2^{-1}f_1^{-1}\\
		&=g_1\left( (f_2,g_2)\bullet A\right) f_1^{-1}\\
		&=(f_1,g_1)\bullet \left( (f_2,g_2)\bullet A\right)
		\end{align*}
		and hence $\bullet$ defines a group action.
	\end{proof}

	Let $A,B\in \Mat_{n\times m}(\mathbb{F})$. The matrices $A$ and $B$ are in the same equivalence class under the group action $\bullet$ if and only if there exists some $(g,f)\in\GL_n(\mathbb{F})\times \GL_m(\mathbb{F})$ such that $(g,f)\bullet A= B$. We let $\langle A \rangle$ denote the equivalence class of $A$ under this equivalence relation. Note that the stabilizer of $A$, $(\GL_n(\mathbb{F})\times \GL_m(\mathbb{F}))_A$, is a subgroup of $\GL_n(\mathbb{F})\times \GL_m(\mathbb{F})$.  
	
	Now we have enough background to state the result for the classifying diagram of the category of finite vector spaces over a field.
	
	\begin{cor} \label{FinVect_stabilizer_perspective}
		The $0$th level of the classifying diagram of $\FinVect( \mathbb{F})$ is given by
		\[N(\FinVect (\mathbb{F}))_0\simeq \coprod_{n\in \mathbb{N}} B(\GL_n(\mathbb{F}))\]
		and the $1$st level is given by 
		\[N(\FinVect (\mathbb{F}))_1\simeq \coprod_{n,m\in \mathbb{N}} \left[ \coprod_{\langle A \rangle\in \Mat_{n\times m}} B\left( (\GL_n\times \GL_m)_A \right) \right] \]
		where $\langle A \rangle$ is the equivalence class  of $A$ under the relation defined by the group action $\bullet$ of $\GL_n\times \GL_m$ acting on $\Mat_{n\times m}(\mathbb{F})$, and $\left( \GL_n\times \GL_m\right)_A$ is the stabilizer of $A$.
	\end{cor}
	
	\begin{proof}
		The characterization of the 0th level follows from Proposition \ref{Bergner 7.2} and the fact that the isomorphisms in $\FinVect (\mathbb{F})$ are given by $\GL_n(\mathbb{F})$ for all non negative integers $n$. Since morphisms from $\mathbb{F}^n$ to $\mathbb{F}^m$ are given by $A\in\Mat_{n\times m}$ and isomorphisms between two $n\times m$ matrices are given by pairs $(f,g)\in \GL_n\times \GL_m$, 
		the desired decomposition of the 1st level of $N(\FinVect(\mathbb{F}))$ follows from Theorem \ref{Stabilizer_perspective}.
	\end{proof}

	We now turn our attention to finding the classifying diagram of a subcategory of $\FinVect(\mathbb{F})$.
\begin{ex}
Let $\FinVect_{\leq 2}(\mathbb{F})$ be the subcategory of $\FinVect(\mathbb{F})$ where the objects are $\mathbb{F}$ and $\mathbb{F}^2$, and the morphisms are given by linear maps. For the $0$th level of the classifying diagram, we get $B(\GL_1(\mathbb{F}))\coprod B(\GL_2(\mathbb{F}))$. Observe that $B(\GL_1(\mathbb{F}))=B(\mathbb{F}^\times )$.

		The $1$st level of the classifying diagram is more interesting.  We break the functor category  $\iso\left( \FinVect_{\leq 2}(\mathbb{F})^{[1]} \right)$ into four types of objects:
		\begin{enumerate}
			\item $\mathbb{F}\to \mathbb{F}$,
			\item $\mathbb{F}\to \mathbb{F}^2$,
			\item $\mathbb{F}^2 \to \mathbb{F}$, and
			\item $\mathbb{F}^2 \to \mathbb{F}^2$.
		\end{enumerate}
		If two objects $\mathbb{F}^i\to \mathbb{F}^j$ and $\mathbb{F}^n\to \mathbb{F}^m$ are isomorphic, then $i=n$ and $j=m$. Note that an isomorphism between two objects $A,B:\mathbb{F}^n\to \mathbb{F}^m$ is given by a pair of matrices $(C,D)$ in $\GL_n(\mathbb{F})\times \GL_m (\mathbb{F})$ such that the diagram
		\[ 
		\begin{tikzcd}
		\mathbb{F}^n \ar[r, "C", "\cong"'] \ar[d, "A"] & \mathbb{F}^n \ar[d, "B"]\\
		\mathbb{F}^m \ar[r, "D", "\cong"'] & \mathbb{F}^m
		\end{tikzcd}
		\]
		commutes. Observe that the objects $A$ and $B$ in the category $\iso\left( \FinVect_{\leq 2}(\mathbb{F})^{[1]} \right)$ are $m\times n$ matrices with entries in $\mathbb{F}$.
		
		\emph{Type (i).}  There are only two $1\times 1$ equivalence classes of matrices with entries in $\mathbb{F}$ under the action $\bullet$, namely, $\langle [0] \rangle=\{A\in \Mat_{1\times 1}(\mathbb{F}): rank (A)=0 \}$ and $\langle [1] \rangle = \{A\in \Mat_{1\times 1}(\mathbb{F}): rank (A)=1 \}$. By Corollary \ref{FinVect_stabilizer_perspective}, it suffices to consider the stabilizer of a representative from each equivalence class. Note that $(C,D)\in \GL_1\times \GL_1 $ is an automorphism for $[0]$ for any pair $(C,D)$ in $\GL_1\times \GL_1$. Also observe that the automorphisms for $[1]$ are of the form $(C,C^{-1})$ in $\GL_1\times \GL_1$.   So type (i) objects contribute $B(\GL_1\times \GL_1)\coprod B(\GL_1)\simeq B(\mathbb{F}^\times \times \mathbb{F}^\times)\coprod B(\mathbb{F}^\times)$ to the first level of the classifying diagram.

		\emph{Type (ii).} 	Under the action $\bullet$, we have the two equivalence classes
		\[\left\langle \left[ \begin{array}{cc}
		0  \\
		0 
		\end{array} \right] \right\rangle
		=
		\left\{ \left[ \begin{array}{cc}
		0  \\
		0 
		\end{array} \right] \right\}
		= \{A\in \Mat_{2\times 1}(\mathbb{F}): rank (A)=0 \}
		\]
		and 
		\[\left\langle \left[ \begin{array}{cc}
		1  \\
		0 
		\end{array} \right] \right\rangle
		= \{A\in \Mat_{2\times 1}(\mathbb{F}): rank (A)=1 \}.
		\]
		To verify, for example, that 	\[
		\left[ \begin{array}{cc}
		0 \\
		1 
		\end{array} \right] \in
		\left\langle \left[ \begin{array}{cc}
		1  \\
		0 
		\end{array} \right] \right\rangle,
		\]
		observe that
		\begin{equation}\label{Group_Action_Eqiv_Class_Example}
		\left( \left[ \begin{array}{cc}
		1
		\end{array} \right],
		\left[ \begin{array}{cc}
		0 & 1 \\
		1 & 0 
		\end{array} \right]\right) \bullet
		\left[ \begin{array}{cc}
		1 \\
		0 
		\end{array} \right] =
		\left[ \begin{array}{cc}
		0 & 1 \\
		1 & 0 
		\end{array} \right]
		\left[ \begin{array}{cc}
		1 \\
		0 
		\end{array} \right]
		\left[ \begin{array}{cc}
		1
		\end{array} \right]^{-1}=
		\left[ \begin{array}{cc}
		0 \\
		1 
		\end{array} \right].
		\end{equation}
		By Corollary \ref{FinVect_stabilizer_perspective}, it suffices to consider the stabilizer of a representative from each equivalence class. Using matrix multiplication, one can check that
		\[ (\GL_1\times \GL_2)_{{\footnotesize
				\left[ \begin{array}{cc}
				0  \\
				0 
				\end{array} \right] }
		}
		= \GL_1\times \GL_2 \cong \mathbb{F}^\times \times \GL_2,
		\]
		and so it remains to find the stabilizer for ${\footnotesize
				\left[ \begin{array}{cc}
				1  \\
				0 
				\end{array} \right] }$
		 in $\GL_1\times \GL_2$.
		
\begin{prop}
The stabilizer for 	${\footnotesize
				\left[ \begin{array}{cc}
				1  \\
				0 
				\end{array} \right] }$
	over $\mathbb{F}$ is given by
		\[ \left(\GL_1  \times \GL_2 \right)_{{\footnotesize
				\left[ \begin{array}{cc}
				1  \\
				0 
				\end{array} \right] }
		}
		= \left\{ 
		\left( [a],
		\left[ \begin{array}{cc}
		a & x  \\
		0 & z
		\end{array} \right] \right):
		a,z\in \mathbb{F}^\times, x\in\mathbb{F}
		\right\}
		\cong \mathbb{F}^\times \times \mathbb{F}^\times \times \mathbb{F},
		\]
		where $\mathbb{F}^\times$ are the units in $\mathbb{F}$.
\end{prop}		

	\begin{proof}
	A pair $(C,D)$ is in $\left(\GL_1  \times \GL_2 \right)_{{\footnotesize
				\left[ \begin{array}{cc}
				1  \\
				0 
				\end{array} \right] }
		}$ 
		if and only if $D \left[ \begin{array}{c}
				1  \\
				0 
				\end{array} \right]
				=
				\left[ \begin{array}{c}
				1  \\
				0 
				\end{array} \right]
				C$.  
	Let $C=\left[ \begin{array}{c}
				a   \\
				\end{array} \right]$ and
				$D=\left[ \begin{array}{cc}
				w & x  \\
				y & z
				\end{array} \right]$. 
	Then $(C,D)$ is in $\left(\GL_1  \times \GL_2 \right)_{{\footnotesize
				\left[ \begin{array}{cc}
				1  \\
				0 
				\end{array} \right] }
		}$  if and only if
	$\left[ \begin{array}{c}
				w   \\
				y 
				\end{array} \right]=
				\left[ \begin{array}{c}
				a   \\
				0 
				\end{array} \right]$, 
	which guarantees that $a=w$ and $y=0$. Also observe that the $a,z\in \mathbb{F}^\times$ to guarantee that the matrices $C$ and $D$ are invertible.	
	\end{proof}

		Thus type (ii) objects contribute $B(\mathbb{F}^\times \times \GL_2)\coprod B(\mathbb{F}^\times \times \mathbb{F}^\times \times \mathbb{F})$ to the first level of the classifying diagram.

		\emph{Type (iii).}  Using properties of the transpose of matrices, we get that $(C,D)\in \left( \GL_1\times \GL_2 \right)_A$ if and only if $(D^T,C^T)\in \left( \GL_2\times \GL_1\right)_{A^T}$. Thus, just like type (ii), type (iii) also contributes $B(\mathbb{F}^\times \times \GL_2)\coprod B(\mathbb{F}^\times \times \mathbb{F}^\times \times \mathbb{F})$ to the first level of the classifying diagram.
		
		\emph{Type (iv).}  Under the group action $\bullet$, there are three equivalence classes:
		
		\[\left\langle \left[ \begin{array}{cc}
		0 & 0 \\
		0 & 0
		\end{array} \right] \right\rangle
		=
		\left\{ \left[ \begin{array}{cc}
		0 & 0  \\
		0 & 0
		\end{array} \right] \right\}
		= \{A\in \Mat_{2\times 2}(\mathbb{F}): rank (A)=0 \},
		\]
		
		\[\left\langle \left[ \begin{array}{cc}
		1 & 0 \\
		0 & 0
		\end{array} \right] \right\rangle
		= \{A\in \Mat_{2\times 2}(\mathbb{F}): rank (A)=1 \}, \]
		and 
		\[\left\langle \left[ \begin{array}{cc}
		1 & 0 \\
		0 & 1
		\end{array} \right] \right\rangle
		= \GL_2(\mathbb{F})
		= \{A\in \Mat_{2\times 2}(\mathbb{F}): rank (A)=2 \}.
		\]
		Note that to verify the above equivalence classes, we use a similar process as (\ref{Group_Action_Eqiv_Class_Example}).  Using matrix multiplication, one can check that the stabilizers of representatives from each equivalence class are given by
		
		\[ (\GL_2\times \GL_2)_{{\footnotesize
				\left[ \begin{array}{cc}
				0 & 0  \\
				0 & 0
				\end{array} \right] }
		}
		= \GL_2 \times \GL_2,
		\]
		and
		\[ (\GL_2\times \GL_2)_{{\footnotesize
				\left[ \begin{array}{cc}
				1 & 0  \\
				0 & 1
				\end{array} \right] }
		}
		= \left\{ 
		\left(C,D\right) \in \GL_2\times \GL_2: D=C^{-1}
		\right\}
		\cong \GL_2,
		\]
		and so it remains to find the stabilizer for $\left[ \begin{array}{cc}
				1 & 0  \\
				0 & 0
				\end{array} \right]$. 

\begin{prop}
The stabilizer for 	${\footnotesize
				\left[ \begin{array}{cc}
				1 & 0  \\
				0 & 0
				\end{array} \right] }$
	over the $\mathbb{F}$ is given by
		\[ \left(\GL_2  \times \GL_2  \right)_{{\footnotesize
				\left[ \begin{array}{cc}
				1 & 0  \\
				0 & 0
				\end{array} \right] }
		}
		= \left\{ 
		\left( 
		\left[ \begin{array}{cc}
		a & b  \\
		0 & d
		\end{array} \right],
		\left[ \begin{array}{cc}
		a & 0  \\
		y & z
		\end{array} \right] \right):
		a,b\in \mathbb{F}^\times, c\in\mathbb{F}
		\right\}
		\]
		\[
		\cong \mathbb{F}^\times \times \mathbb{F}^\times \times \mathbb{F}^\times \times \mathbb{F}\times \mathbb{F},		
		\]
		where $\mathbb{F}^\times$ are the units in $\mathbb{F}$.
\end{prop}

\begin{proof}
	A pair of matrices  $(C,D)$ is in $\left(\GL_2  \times \GL_2  \right)_{{\footnotesize
				\left[ \begin{array}{cc}
				1 & 0  \\
				0 & 0
				\end{array} \right] }
		}$ if and only if $D \left[ \begin{array}{cc}
				1 & 0  \\
				0 & 0
				\end{array} \right]
				=
				\left[ \begin{array}{cc}
				1 & 0  \\
				0 & 0
				\end{array} \right]
				C$.  
	Let $C=\left[ \begin{array}{cc}
				a & b  \\
				c & d
				\end{array} \right]$ and
				$D=\left[ \begin{array}{cc}
				w & x  \\
				y & z
				\end{array} \right]$. 
	Then $(C,D)$ is in $\left(\GL_2  \times \GL_2  \right)_{{\footnotesize
				\left[ \begin{array}{cc}
				1 & 0  \\
				0 & 0
				\end{array} \right] }
		}$ if and only if
	$\left[ \begin{array}{cc}
				a & 0  \\
				c & 0
				\end{array} \right]=
				\left[ \begin{array}{cc}
				w & x  \\
				0 & 0
				\end{array} \right]$, 
	which guarantees that $a=w$, $c=0$, and $x=0$. Also observe that $a,d,z\in \mathbb{F}^\times$ to guarantee that the matrices $C$ and $D$ are invertible.
\end{proof}

Thus type (iv) contributes $B(\GL_2\times \GL_2) \coprod B(\mathbb{F}^\times \times \mathbb{F}^\times \times \mathbb{F}^\times \times \mathbb{F} \times \mathbb{F}) \coprod B(\GL_2)$ to the first level of the classifying diagram.

		Putting together what we obtain from all four types of objects in $\iso\left( \FinVect_{\leq 2}(\mathbb{F})^{[1]} \right)$, we can describe the first level of the classifying diagram as 
		\begin{align*}
		N(\FinVect_{\leq 2}(\mathbb{F}) )_1 \simeq B(& \mathbb{F}^\times \times  \mathbb{F}^\times) \coprod B(\mathbb{F}^\times)\\
		&    \coprod B(\mathbb{F}^\times \times \GL_2)\coprod B(\mathbb{F}^\times \times \mathbb{F}^\times \times \mathbb{F}) \\
		&    \coprod B(\mathbb{F}^\times \times \GL_2)\coprod B(\mathbb{F}^\times \times \mathbb{F}^\times \times \mathbb{F}) \\
		& \coprod B(\GL_2\times \GL_2) \coprod B(\mathbb{F}^\times \times \mathbb{F}^\times \times \mathbb{F}^\times \times \mathbb{F} \times \mathbb{F}) \coprod B(\GL_2).
		\end{align*}

\end{ex}

	\section{The classifying diagram for the category of finite sets}\label{Section_FinSet}
	
	In this section we prove that the classifying diagram for the category of finite sets, denoted by $\FinSet$, can be decomposed into the classifying spaces of products of wreath products. 
	As a consequence, we also prove decompositions of the subcategories $\FinSet_{inj}$ and $\FinSet_{surj}$, the subcategories consisting of injective and surjective functions, respectively. In order to work with the $1$st level of the classifying diagram, as we have seen previously, we need to understand automorphisms of morphisms; each function between finite sets can be depicted by a tree. The section begins with the definition of a wreath product and recalling the relationship between wreath products and automorphisms of trees.
	
	\subsection{Wreath products and trees}
	
	Let us recall the definition of the wreath product.  Let $K$ and $L$ be two groups and $\rho:K\to \Sigma_n$ be a homomorphism where $\Sigma_n$ is the $n$th symmetric group. Let $H\defeq L^n$; 
	an injective homomorphism $\phi: \Sigma_n\to Aut(H)$ can be constructed by letting the elements of $\Sigma_n$ permute the $n$ factors of $H$. The \emph{wreath product} of $L$ by $K$, denoted by $L\wr K$, is the semidirect product $H\rtimes K$ with respect to the homomorphism $\phi\circ \rho:K\to Aut(H)$ \cite[\S 5.5, Ex. 23]{DummitFoote}. 
	Wreath products are nice tools for describing the group of automorphisms of specific types of rooted trees. We recall definitions relevant to trees.
	
	\begin{defn}
		\begin{enumerate}
			\item A \emph{rooted tree} is a connected simple graph without cycles and with a distinguished vertex called the \emph{root}.
			\item  A vertex $u$ is \emph{adjacent} to a vertex $v$ in a tree if there is an edge between $u$ and $v$. 
			\item The \emph{level} of a vertex $v$ in a rooted tree is the length of the unique path from the root to this vertex.
			\item If $u$ is a vertex at level $j$ that is adjacent to a vertex $v$ at level $j + 1$, then $v$ is said to be a \emph{child} of $u$ and $u$ is the \emph{parent} of $v$. 
			\item The \emph{height} of a rooted tree is the length of the longest path from the root to any vertex.
			\item Let $V$ be the set of vertices for a rooted tree $\Gamma$. An \emph{automorphism} of $\Gamma$ is a bijection $\phi:V\to V$ such that $u$ and $v$ are adjacent if and only if $\phi(u)$ and $\phi(v)$ are adjacent. 
		\end{enumerate}
		\end{defn}
	
	We only consider trees of height 2. To see how wreath products are used to describe automorphisms on rooted trees of height 2, consider the tree 
	
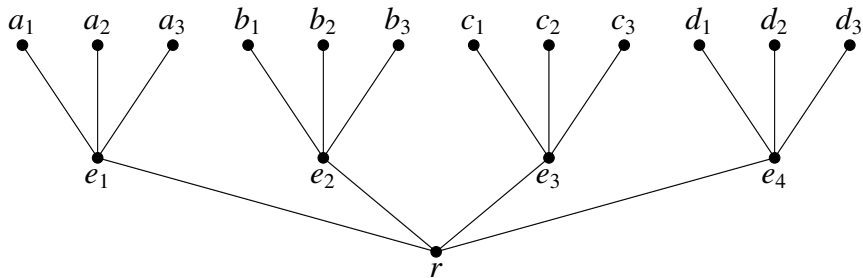
\begin{center}
	\begin{tikzpicture}
	\filldraw 
	(1,2.5) circle (2pt) node[above] {$a_1$} --
	(2,1) circle (2pt) node[below] {$e_1$}     -- 
	(6.5,-.25) circle (2pt) node[below] {$r$}
	(2,2.5) circle (2pt) node[above] {$a_2$} --	(2,1)
	(3,2.5) circle (2pt) node[above] {$a_3$} --	(2,1)
	(4,2.5) circle (2pt) node[above] {$b_1$} --
	(5,1) circle (2pt) node[below] {$e_2$}     -- (6.5,-.25)
	(5,2.5) circle (2pt) node[above] {$b_2$} --	(5,1)
	(6,2.5) circle (2pt) node[above] {$b_3$} --	(5,1)
	(7,2.5) circle (2pt) node[above] {$c_1$} --
	(8,1) circle (2pt) node[below] {$e_3$}     -- (6.5,-.25)
	(8,2.5) circle (2pt) node[above] {$c_2$} --	(8,1)
	(9,2.5) circle (2pt) node[above] {$c_3$} --	(8,1)
	(10,2.5) circle (2pt) node[above] {$d_1$} --
	(11,1) circle (2pt) node[below] {$e_4$}     -- (6.5,-.25)
	(11,2.5) circle (2pt) node[above] {$d_2$} --	(11,1)
	(12,2.5) circle (2pt) node[above] {$d_3$} --	(11,1)
	 ;
	\end{tikzpicture}
	\end{center}
%
%
	with root $r$; each level 1 vertex is the parent of 3 children.  An automorphism on this tree has two different group actions on the vertices.  First, we have a $\Sigma_4$ action occurring on the vertices $\{e_1,e_2,e_3,e_4\}$.  We also have four different $\Sigma_3$ actions occurring; $\Sigma_3$ acts on the children of each $e_i$.  For example, we have a $\Sigma_3$ action on $\{a_1,a_2,a_3\}$.  This is a good example for what a wreath product captures.  The automorphisms on the above tree are described be the wreath product of $\Sigma_3$ by $\Sigma_4$, or using the wreath notation, it is $\Sigma_3\wr\Sigma_4$. We denote the above tree as $\Gamma_{3,4}$.
	
	The more general result also holds.  Let $\Gamma_{n,m}$ denote the rooted tree where the root has children $e_1,\ldots,e_m$ and each $e_i$ has $n$ children.  The group of automorphisms $Aut(\Gamma_{n,m})$ is isomorphic to $\Sigma_n\wr \Sigma_m$.
	
	\subsection{The classifying diagram}
	
	So what is the purpose of talking about automorphism on trees like $\Gamma_{n,m}$? We can use these trees to describe a morphism from a set of order $n$ to a set of order $m$; in the first level of the classifying diagram, we need to describe the group of automorphisms for a morphism.  The above pictured tree, $\Gamma_{3,4}$, is an example of how a set of order $12$ can map to a set of order $4$. But there are many other ways for a set of order $12$ to map to a set of order $4$. For example, the tree
		
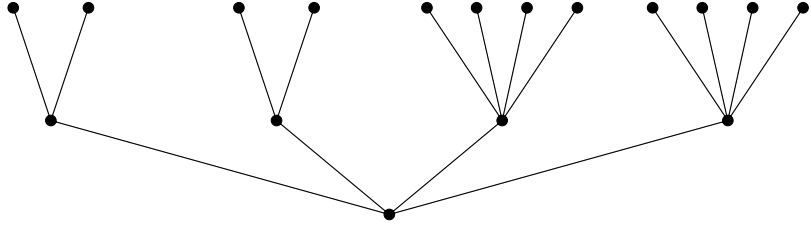
\begin{center}
		\begin{tikzpicture}
		\filldraw 
		(1.5,2.5) circle (2pt) node[above] {} --
		(2,1) circle (2pt) node[below] {}     -- 
		(6.5,-.25) circle (2pt) node[below] {}
		(2.5,2.5) circle (2pt) node[above] {} --	(2,1)
		(4.5,2.5) circle (2pt) node[above] {} --
		(5,1) circle (2pt) node[below] {}     -- (6.5,-.25)
		(5.5,2.5) circle (2pt) node[above] {} --	(5,1)
		(7,2.5) circle (2pt) node[above] {} --
		(8,1) circle (2pt) node[below] {}     -- (6.5,-.25)
		(7.66,2.5) circle (2pt) node[above] {} --	(8,1)
		(8.33,2.5) circle (2pt) node[above] {} --	(8,1)
		(9,2.5) circle (2pt) node[above] {} --	(8,1)
		(10,2.5) circle (2pt) node[above] {} --
		(11,1) circle (2pt) node[below] {}     -- (6.5,-.25)
		(10.66,2.5) circle (2pt) node[above] {} --	(11,1)
		(11.33,2.5) circle (2pt) node[above] {} --	(11,1)
		(12,2.5) circle (2pt) node[above] {} --	(11,1)
		;
		\end{tikzpicture}
	\end{center}
	is another way for a set of order $12$ to map to a set of $4$; the group of automorphisms is now given by a product of wreath products: $\left(\Sigma_2\wr \Sigma_2\right) \times \left(\Sigma_4\wr \Sigma_2\right)$. We denote this latter tree as $\Gamma_{2,2}\cup\Gamma_{4,2}$. 
	
	In general, given two rooted trees $\Gamma_1$ and $\Gamma_2$, the rooted tree $\Gamma_1\cup\Gamma_2$ is given by the trees $\Gamma_1$ and $\Gamma_2$ identified at the root.
		
	Let us revisit the trees $\Gamma_{3,4}$ and $\Gamma_{2,2}\cup \Gamma_{4,2}$; these trees depict functions from a set of order $12$ to a set of order $4$. Let $k_i$ be the number of children of the root that has $i$ children. For the tree $\Gamma_{3,4}$, $(k_0,k_1,k_2,k_3,k_4)=(0,0,0,4,0)$, and for the tree $\Gamma_{2,2}\cup \Gamma_{4,2}$ $(k_0,k_1,k_2,k_3,k_4)=(0,0,2,0,2)$. 
	In either case, and in general for the tree associated with an arbitrary morphism going from a set of order $12$ to a set of order $4$, the equations $k_1+2k_2+3k_3+4k_4=12$ and $k_0+k_1+k_2+k_3+k_4=4$ are satisfied. We get an analogous system of equations if we instead considered functions from a set of order $n$ to a set of order $m$.

	Now that we have set up the desired notation for trees and their correspondence with functions between finite sets, we can state the result for the classifying diagram of $\FinSet$.
	
	\begin{thm} \label{FinSet_Classifying_Diagram}
		The $0$th level of the classifying diagram of $\FinSet$ is given by
		\[N\left(\FinSet\right)_0\simeq \coprod_{n\in\mathbb{N}}B(\Sigma_n).\]
		And the $1$st level is described as 
		\[N \left(\FinSet\right)_1  \simeq \coprod B\left( \left( \Sigma_0 \wr \Sigma_{k_0} \right) \times \cdots \times \left(\Sigma_n \wr \Sigma_{k_n}\right) \right)\]
		where the disjoint union is over solutions to the equations
		\begin{center}
			\begin{tabular}{l c r}
				$k_1+2k_2+\cdots +nk_n=n$ & and & $k_0+k_1+\cdots +k_n=m$
			\end{tabular}
		\end{center}
		given that $n,m,k_0,\ldots,k_n$ are non negative integers.
	\end{thm}
	
	\begin{proof}
		We begin by starting with the $0$th level of the classifying diagram. The functor \[\iota:\coprod_n\Sigma_n\to\iso(\FinSet),\] which is given by sending, for each $n$, the subcategory $\Sigma_n$ to a set of order $n$ defines an equivalence of categories. Applying the nerve functor gives the desired result.

		Now we consider the $1$st level of the classifying diagram.  Note that the objects in $\iso(\FinSet^{[1]})$ are functions between finite sets, $f:A\to B$, where $A$ is a set of order $n$ and $B$ is a set of order $m$. Let $k_i$ be the number of elements in $B$ whose preimage under $f$ has cardinality $i$. The function $f:A\to B$ can be identified with the tree $\Gamma_{0,k_0}\cup \Gamma_{1,k_1} \cup \cdots \cup \Gamma_{n,k_n}$
		, which has automorphism group $\left( \Sigma_0 \wr \Sigma_{k_0} \right) \times \cdots \times \left(\Sigma_n \wr \Sigma_{k_n}\right)$. For a fixed $n$ and $m$, note that there is a bijective correspondence between automorphism classes of $f:A\to B$ and nonnegative integer solutions to the equations
		\begin{center}
			\begin{tabular}{l c r}
				$k_1+2k_2+\cdots +nk_n=n$ & and & $k_0+k_1+\cdots +k_n=m$.
			\end{tabular}
		\end{center}
		Thus for a fixed $n$ and $m$, \[\coprod B\left( \left( \Sigma_0 \wr \Sigma_{k_0} \right) \times \cdots \times \left(\Sigma_n \wr \Sigma_{k_n}\right) \right),\]
		where the disjoint union is over solutions to the above equations, captures the contributions to the first level of the classifying diagram of functions from sets of order $n$ to sets of order $m$.  Ranging over all possible $n$ and $m$ gives the desired result.
	\end{proof}

	Observe that the diagram 
	\[
	\begin{tikzcd}
	& B\left( \left( \Sigma_0 \wr \Sigma_{k_0} \right) \times \cdots \times \left(\Sigma_n \wr \Sigma_{k_n}\right) \right) \ar[dl,"d_1"'] \ar[dr,"d_0"] &\\
	B(\Sigma_n) & & B(\Sigma_m)
	\end{tikzcd}
	\]
	shows how the face maps  $d_0,d_1:N(\FinSet)_1\to N(\FinSet)_0$ interact where the equations
	\begin{center}
		\begin{tabular}{l c r}
			$k_1+2k_2+\cdots +nk_n=n$ & and & $k_0+k_1+\cdots +k_n=m$
		\end{tabular}
	\end{center}
	are satisfied.

	Let $\FinSet_{inj}$ denote the subcategory of $\FinSet$ where the objects are the same as $\FinSet$, but the morphisms are restricted to the injective functions.
	
	\begin{cor}
		The zeroth level of the classifying diagram of $\FinSet_{inj}$ is given by
		\[N\left(\FinSet_{inj}\right)_0\simeq \coprod_{n\in\mathbb{N}}B(\Sigma_n).\]
		And the first level is described as 
		\[N \left(\FinSet_{inj} \right)_1  \simeq \coprod_{n,m\in\mathbb{N}, n\leq m}B\left(  \Sigma_n\right).\]
	\end{cor}
	
	\begin{proof}
		The proof for the $0$th level is identical to Theorem \ref{FinSet_Classifying_Diagram}. So it remains to justify the description of the $1$st level. Using the same setup at the proof of \ref{FinSet_Classifying_Diagram} with the exception that we require $f:A\to B$ to be injective and hence $n\leq m$,  the function $f:A\to B$ is identified with the tree $\Gamma_{0,k_0}\cup \Gamma_{1,k_1} \cup \cdots \cup \Gamma_{n,k_n}$ where $k_i=0$ if $i\geq 2$ since $f$ is injective. Following the argument of the proof of \ref{FinSet_Classifying_Diagram},  the system of equations is simplified to just \[k_0+n=m.\] Thus for a fixed $n$ and $m$ where $n\leq m$, \[B(\Sigma_n)\] captures the contributions to the first level of the classifying diagram of injective functions from sets of order $n$ to sets of order $m$. Ranging over all possible $n$ and $m$, where $n\leq m$, gives the desired result.
	\end{proof}
	
	Let $\FinSet_{surj}$ denote the subcategory of $\FinSet$ where the objects are the same as $\FinSet$, but the morphisms are restricted to surjective functions.
	
	\begin{cor}
		The zeroth level of the classifying diagram of $\FinSet$ is given by
		\[N\left(\FinSet_{surj}\right)_0\simeq \coprod_{n\in\mathbb{N}}B(\Sigma_n).\]
		And the first level is described as 
		\[N \left(\FinSet_{surj}\right)_1  \simeq \coprod B\left( \Sigma_{k_1} \times \left( \Sigma_2 \wr \Sigma_{k_2} \right) \times \cdots \times \left(\Sigma_n \wr \Sigma_{k_n}\right) \right)\]
		where the disjoint union is over solutions to the equations
		\begin{center}
			\begin{tabular}{l c r}
				$k_1+2k_2+\cdots +nk_n=n$ & and & $k_1+\cdots +k_n=m$
			\end{tabular}
		\end{center}
		given that $n,m,k_0,\ldots,k_n$ are non negative integers and $n\geq m$.
	\end{cor}
	
	\begin{proof}
		The proof of the $0$th level is identical to Theorem \ref{FinSet_Classifying_Diagram}. To get the desired result for the $1$st level, we implement the same setup as the proof of \ref{FinSet_Classifying_Diagram} with the exception that we require $f:A\to B$ to be a surjective function. Since $f$ is surjective, $n\geq m$. Hence $f:A\to B$ is identified with the tree $\Gamma_{0,k_0}\cup \Gamma_{1,k_1} \cup \cdots \cup \Gamma_{n,k_n}$ where $k_0=0$ since $f$ is surjective. Following the same argument of the proof of \ref{FinSet_Classifying_Diagram} with the added restrictions that $k_0=0$ and $n\geq m$ gives the desired result.
	\end{proof}

	\appendix


\end{document}